\documentclass[leqno]{amsart}

\usepackage{latexsym}
\usepackage[all]{xy}

\usepackage{amssymb} 
\usepackage{amsmath} 
\usepackage{color}
\usepackage{comment}
\usepackage{mathdots}
\usepackage{bm}

\usepackage{tikz} 
\usepackage{subcaption} 

\definecolor{gray}{gray}{0.7}
\definecolor{Gray}{gray}{0.3}

\usepackage{amscd}

\def\Z{{\mathbb{Z}}}
\def\R{{\mathbb{R}}}
\def\C{{\mathbb{C}}}

\def\Y{Y}

\def\m{m}
\def\wh{w_h}

\DeclareMathOperator{\Fl}{Flag}
\DeclareMathOperator{\Hess}{Hess}

\DeclareMathOperator{\id}{id}
\DeclareMathOperator{\vol}{Vol}
\DeclareMathOperator{\GZ}{GZ}
\DeclareMathOperator{\GL}{GL}
\DeclareMathOperator{\SL}{SL}

\DeclareMathOperator{\Perm}{Perm} 

\numberwithin{equation}{section}

\theoremstyle{plain} 
 \newtheorem{theorem}{Theorem}[section]
 \newtheorem{proposition}[theorem]{Proposition}
 \newtheorem{corollary}[theorem]{Corollary}
 \newtheorem{lemma}[theorem]{Lemma}
 
\newtheorem{problem}[theorem]{Problem}

 \theoremstyle{definition}
 
 \newtheorem{remark}[theorem]{Remark}
 \newtheorem{example}[theorem]{Example}

\title[positivity of volume polynomial]{The volume polynomial of regular semisimple Hessenberg varieties and the Gelfand-Zetlin polytope}

\author{Megumi Harada}
\address{Department of Mathematics and Statistics, McMaster University, 1280 Main Street West, Hamilton, Ontario L8S4K1, Canada}
\email{Megumi.Harada@math.mcmaster.ca}

\author{Tatsuya Horiguchi}
\address{Department of Pure and Applied Mathematics,
Graduate School of Information Science and Technology,
Osaka University, 1-5 Yamadaoka, Suita, Osaka 565-0871, Japan}
\email{tatsuya.horiguchi0103@gmail.com}

\author{Mikiya Masuda}
\address{Department of Mathematics, Osaka City University, Sumiyoshi-ku, Osaka 558-8585, Japan.}
\email{masuda@sci.osaka-cu.ac.jp}

\author{Seonjeong Park}
\address{Department of Mathematics, Ajou University, Suwon 16499, Republic of Korea}
\email{seonjeong1124@gmail.com}

\date{\today}

\begin{document}

\maketitle

\begin{abstract}
Regular semisimple Hessenberg varieties are subvarieties of the flag variety $\Fl(\C^n)$ arising naturally in the intersection of geometry, representation theory, and combinatorics. Recent results of Abe-Horiguchi-Masuda-Murai-Sato and Abe-DeDieu-Galetto-Harada relate the volume polynomials of regular semisimple Hessenberg varieties to the volume polynomial of the Gelfand-Zetlin polytope $\GZ(\lambda)$ for $\lambda=(\lambda_1,\lambda_2,\ldots,\lambda_n)$. The main results of this manuscript use and generalize tools developed by Anderson-Tymoczko, Kiritchenko-Smirnov-Timorin, and Postnikov, in order to derive an explicit formula for the volume polynomials of regular semisimple Hessenberg varieties in terms of the volumes of certain faces of the Gelfand-Zetlin polytope, and also exhibit a manifestly positive, combinatorial formula for their coefficients with respect to the basis of monomials in the $\alpha_i := \lambda_i-\lambda_{i+1}$. In addition, motivated by these considerations, we carefully analyze the special case of the permutohedral variety, which is also known as the toric variety associated to Weyl chambers. In this case, we obtain an explicit decomposition of the permutohedron (the moment map image of the permutohedral variety) into combinatorial $(n-1)$-cubes, and also give a geometric interpretation of this decomposition by expressing the cohomology class of the permutohedral variety in $\Fl(\C^n)$ as a sum of the cohomology classes of a certain set of Richardson varieties. 
\end{abstract}

\bigskip

\tableofcontents

\section{Introduction}

In this manuscript we study the volume polynomials of Hessenberg varieties and other subvarieties of the flag variety $\Fl(\C^n)$, with particular attention to the relationship between these polynomials and the volume polynomials of related polytopes. The terminology used in this circle of ideas can be confusing, so we take a moment to clarify the meaning of these terms. The main point is that, on the one hand, there is a notion (made precise in Section~\ref{sec: volume}) of a volume polynomial of a cohomology class or, in the context of this paper, a subvariety $Y$ of $\Fl(\C^n)$. This notion is closely related to that of \emph{degree} in the sense of algebraic geometry, and thus encodes information about the topology of $Y$. For the purpose of this introduction only, we refer to this concept as the \emph{geometric} volume polynomial. \footnote{As far as we are aware, this is not standard terminology, so we avoid use of this term elsewhere in the manuscript.} On the other hand, given a convex polytope $\Delta$ in a Euclidean space $\R^n$, we also have the usual notion of the (Euclidean, or more accurately, an appropriately normalized integral Euclidean) volume of $\Delta$. For the purpose of this introduction only, we refer to this as the \emph{combinatorial} volume. A priori, these two notions of volume are unrelated. However, when the ambient space is the flag variety $\Fl(\C^n)$ and the polytope in question is the Gelfand-Zetlin polytope $\GZ(\lambda)$, then there are beautiful relationships between the geometric volume polynomials of various subvarieties of $\Fl(\C^n)$ and the combinatorial volumes of certain faces of $\GZ(\lambda)$. 
We are not the first to observe or exploit these relationships; indeed, part of the motivation for this paper was the work of Kiritchenko, Smirnov, and Timorin and others (cf. \cite{KST12} and references therein) which relate the cohomology classes of Schubert varieties and products of Schubert varieties with certain faces of $\GZ(\lambda)$. The new direction explored in this manuscript is the focus on the Hessenberg (sub)varieties in $\Fl(\C^n)$.

Hessenberg varieties are a family of subvarieties of the flag variety that arise naturally in  combinatorics, geometry, and representation theory, and is an active area of current research, in part due to the connections between the regular semisimple Hessenberg varieties $\Hess(S,h)$ (defined precisely in Section~\ref{sec: background}) and the famous Stanley-Stembridge conjecture in graph theory and the theory of symmetric functions (see e.g. \cite{ShareshianWachs2016, BrosnanChow2015, Guay-Paquet2016, HaradaPrecup2017}). 
Most recently, volume polynomials also appeared in the study of Hessenberg varieties through the work of the second and third authors together with Abe, Murai, and Sato \cite{AHMMS}, in which they use the theory of hyperplane arrangements to derive explicit generators and relations for the cohomology rings of regular nilpotent Hessenberg varieties in certain Lie types. In related work, the first author together with Abe, DeDieu, and Galetto studied the volume polynomials of both regular semisimple and regular nilpotent Hessenberg varieties in relation to the study of Newton-Okounkov bodies \cite{ADGH}. It can be seen from the explicit formula for the volume polynomials of regular nilpotent Hessenberg varieties given in \cite{AHMMS} that it is intimately related to the volume of Gelfand-Zetlin polytopes. Therefore, the time seemed ripe for a careful study of the volume polynomials of Hessenberg and related subvarieties. In this manuscript, we focus exclusively on the case of the regular semisimple Hessenberg varieties $\Hess(S,h)$.

We now explain the results contained in this manuscript in some more detail. Since the notation and terminology is technical, we only give rough statements here; precise statements can be found in the body of the manuscript. Let $\lambda = (\lambda_1, \lambda_2, \cdots, \lambda_n) \in \R^n$, where $\lambda_1 > \lambda_2 > \cdots > \lambda_n$. The Gelfand-Zetlin polytope $\GZ(\lambda)$ is a polytope in $\R^{n(n-1)/2}$ defined by certain inequalities depending on the parameter $\lambda$ (see Section~\ref{sec: background}). Thus, the combinatorial volume of $\GZ(\lambda)$ may be viewed as a polynomial in the variables $\lambda_1, \ldots, \lambda_n$. It turns out that the geometric volume polynomial $\vol_\lambda(\Hess(S,h))$ of a regular semisimple Hessenberg variety in $\Fl(\C^n)$ can be obtained by taking certain derivatives, with respect to the $\lambda_i$'s, of the combinatorial volume polynomial $\vol(\GZ(\lambda))$ of $\GZ(\lambda)$ \cite{ADGH, AHMMS}. 
Straightforward computations in small-$n$ cases lead us to believe that this geometric volume $\vol_\lambda(\Hess(S,h))$ can alternately be expressed as a non-negative linear combination of the combinatorial volumes of certain faces of $\GZ(\lambda)$. We were thus lead to ask:
\begin{itemize}
\item Is there an explicit formula expressing $\vol_\lambda(\Hess(S,h))$ as a sum of volumes of faces of $\GZ(\lambda)$?
\item Is there a manifestly positive, combinatorial formula for the coefficients of $\vol_\lambda(\Hess(S,h))$ when expressed in the basis of monomials in the $\alpha_i:=\lambda_i - \lambda_{i+1}$? 
\end{itemize} 
We record positive answers to the above two questions in Section~\ref{sec: volume}. It turns out that the answer to the first question can be obtained as a straightforward consequence of results of Anderson-Tymoczko and Kiritchenko-Smirnov-Timorin \cite{AT, KST12}, 
as we record in Theorem~\ref{thm: vol Hess 1}. In addition, we prove the following. 

\begin{enumerate} 
\item In Section~\ref{sec: combinatorics} we prove two combinatorial results about the faces of Gelfand-Zetlin polytopes. The first result is a generalization of a result of Postnikov \cite{Post}, and is a manifestly positive, combinatorial formula for the coefficients of the combinatorial volume of any face $F$ of $\GZ(\lambda)$ in the basis of monomials in the $\alpha_i=\lambda_i-\lambda_{i+1}$ (Proposition~\ref{prop:face of GZ}). The second result is a generalization of a formula of Kiritchenko, Timorin, and Smirnov, and is a linear relation between the volumes of several closely related faces of $\GZ(\lambda)$ (Proposition~\ref{prop:x-relation}). 
\item In Section~\ref{sec: volume} we use the results in Section~\ref{sec: combinatorics} to answer the second question above, namely, we give a manifestly positive formula for the coefficients of $\vol_\lambda(\Hess(S,h))$ with respect to the basis of monomials in the $\alpha_i$'s (Theorem~\ref{thm:vol_Hess 2}). 
\item Our considerations in Section~\ref{sec: volume} motivated us to look at the special case of the permutohedral variety, which is a special case of a regular semisimple Hessenberg variety and which is also known as the toric variety associated to the Weyl chambers. From standard symplectic geometry it is well-known that the moment map image of the permutohedral variety is the so-called permutohedron. In Section~\ref{sec:decomposing} we use a certain well-chosen subset of the faces of $\GZ(\lambda)$ to give a decomposition of the permutohedron into combinatorial cubes (Theorem~\ref{prop:perm-cubes}). 
\item As noted above, the permutohedron has a geometric counterpart, namely, the permutohedral variety. In Section~\ref{section: Richardson} we show that our combinatorial result in Section~\ref{sec:decomposing} has a geometric interpretation, namely, that the cohomology class of the permutohedral variety in $\Fl(\C^n)$ is a sum of classes, where the summands are in one-to-one correspondence with the cubes appearing in the decomposition given in Theorem~\ref{prop:perm-cubes}. 
\end{enumerate}

We now outline the contents of the paper.  In Section~\ref{sec: background} we briefly recount the necessary background:  Hessenberg varieties, Schubert varieties, and some combinatorics associated to Gelfand-Zetlin polytopes. In Section~\ref{sec: combinatorics} we derive some combinatorial results concerning the faces of Gelfand-Zetlin polytopes. In addition to being of interest in their own right, these results will be used in the arguments in the next sections. In Section~\ref{sec: volume} we define the (geometric) volume polynomial and answer the two questions -- as noted above -- which motivated this paper. The considerations of the special case of the permutohedral variety are contained in the last two sections: the combinatorial view is given in Section~\ref{sec:decomposing}, and the geometric perspective is discussed in Section~\ref{section: Richardson}.

\bigskip

\noindent \textbf{Acknowledgements.} 
We are grateful to Kiumars Kaveh for useful conversations. We are also grateful to Shizuo Kaji for communicating to us the results of his computer computations (cf. Remark~\ref{remark: kaji}). 
The first author is supported in part by an NSERC Discovery Grant and a Canada Research Chair (Tier 2) Award. She also gratefully acknowledges the support and hospitality of the Osaka City University Advanced Mathematics Institute for a fruitful visit in Fall 2017, during which much of this work was conducted. 
The second author is partially supported by JSPS Grant-in-Aid for JSPS Fellows:17J04330.
The third author is supported in part by JSPS Grant-in-Aid for Scientific Research 16K05152.  
The fourth author acknowledges the support of Basic Science Research Program through the National Research Foundation of Korea (NRF) funded by the Ministry of Science, ICT \& Future Planning (NRF-2018R1A6A3A11047606).

\section{Background on Hessenberg and Schubert varieties} \label{sec: background}

This section contains the necessary background for the remainder of the paper. We begin with a brief outline of its contents. 
We first introduce the main geometric characters appearing in this manuscript, namely, the regular semisimple Hessenberg varieties and the Schubert varieties. Both are subvarieties of the flag variety $\Fl(\C^n)$. We also record a result of Anderson-Tymoczko \cite{AT} which expresses the cohomology\footnote{Throughout this document (unless explicitly stated otherwise) we work with cohomology rings with coefficients in $\R$.} class associated to (i.e., the Poincar\'e dual of) a regular semisimple Hessenberg variety in terms of those of Schubert varieties. We then introduce the main combinatorial object of the paper, namely, the Gelfand-Zetlin polytope and its faces. Next, we recall the results of Kiritchenko, Smirnov, and Timorin which relates the cohomology classes of (products of) Schubert varieties with certain faces of the Gelfand-Zetlin polytope. Finally, we state in Corollary~\ref{cor: Hess AT KST} a formula relating the Poincar\'e dual of the regular semisimple Hessenberg variety to a certain subset of faces of the Gelfand-Zetlin polytope, which we analyze further in Section~\ref{sec: volume}.

We begin with the definition of regular semisimple Hessenberg varieties. 
Let $n$ be a positive integer.
The {\bf (full) flag variety} $\Fl(\C^n)$ in $\C^n$ is the collection of nested linear subspaces
$$
V_{\bullet}:=(V_1 \subset V_2 \subset \cdots \subset V_n=\C^n)
$$
where each $V_i$ is an $i$-dimensional subspace in $\C^n$. Hessenberg varieties are subvarieties of $\Fl(\C^n)$ defined as follows. 
Let us fix the notation 
\[
[n] := \{1,2,\ldots, n\}.
\]
A Hessenberg function is a function $h: [n] \to [n]$ satisfying $h(i+1) \geq h(i)$ for all $i$ with $1 \leq i \leq n-1$ and $h(i) \geq i$ for all $i$ with $1 \leq i \leq n$. We frequently denote a Hessenberg function by listing its values in sequence, i.e., $h=(h(1), h(2), \ldots, h(n))$. We also visualize Hessenberg functions as a collection of boxes as follows. Consider the set of $n \times n$ many boxes arranged as in a square matrix. We color the $(i,j)$-th box if $i \leq h(j)$, or in other words, for each $j$ with $1 \leq j \leq n$, we color $h(j)$ many boxes in the $j$-th column starting from the top. The colored boxes are then a visual representation of $h$. 

\begin{example}\label{example: box diagram} 
Let $n=5$ and let $h=(3,3,4,5,5)$. The corresponding set of colored boxes is illustrated in Figure~\ref{picture:h=(3,3,4,5,5)}. 
\begin{figure}[h]
\begin{center}
\begin{picture}(75,65)
\put(0,63){\colorbox{gray}}
\put(0,67){\colorbox{gray}}
\put(0,72){\colorbox{gray}}
\put(4,63){\colorbox{gray}}
\put(4,67){\colorbox{gray}}
\put(4,72){\colorbox{gray}}
\put(9,63){\colorbox{gray}}
\put(9,67){\colorbox{gray}}
\put(9,72){\colorbox{gray}}

\put(15,63){\colorbox{gray}}
\put(15,67){\colorbox{gray}}
\put(15,72){\colorbox{gray}}
\put(19,63){\colorbox{gray}}
\put(19,67){\colorbox{gray}}
\put(19,72){\colorbox{gray}}
\put(24,63){\colorbox{gray}}
\put(24,67){\colorbox{gray}}
\put(24,72){\colorbox{gray}}

\put(30,63){\colorbox{gray}}
\put(30,67){\colorbox{gray}}
\put(30,72){\colorbox{gray}}
\put(34,63){\colorbox{gray}}
\put(34,67){\colorbox{gray}}
\put(34,72){\colorbox{gray}}
\put(39,63){\colorbox{gray}}
\put(39,67){\colorbox{gray}}
\put(39,72){\colorbox{gray}}

\put(45,63){\colorbox{gray}}
\put(45,67){\colorbox{gray}}
\put(45,72){\colorbox{gray}}
\put(49,63){\colorbox{gray}}
\put(49,67){\colorbox{gray}}
\put(49,72){\colorbox{gray}}
\put(54,63){\colorbox{gray}}
\put(54,67){\colorbox{gray}}
\put(54,72){\colorbox{gray}}

\put(60,63){\colorbox{gray}}
\put(60,67){\colorbox{gray}}
\put(60,72){\colorbox{gray}}
\put(64,63){\colorbox{gray}}
\put(64,67){\colorbox{gray}}
\put(64,72){\colorbox{gray}}
\put(69,63){\colorbox{gray}}
\put(69,67){\colorbox{gray}}
\put(69,72){\colorbox{gray}}

\put(0,48){\colorbox{gray}}
\put(0,52){\colorbox{gray}}
\put(0,57){\colorbox{gray}}
\put(4,48){\colorbox{gray}}
\put(4,52){\colorbox{gray}}
\put(4,57){\colorbox{gray}}
\put(9,48){\colorbox{gray}}
\put(9,52){\colorbox{gray}}
\put(9,57){\colorbox{gray}}

\put(15,48){\colorbox{gray}}
\put(15,52){\colorbox{gray}}
\put(15,57){\colorbox{gray}}
\put(19,48){\colorbox{gray}}
\put(19,52){\colorbox{gray}}
\put(19,57){\colorbox{gray}}
\put(24,48){\colorbox{gray}}
\put(24,52){\colorbox{gray}}
\put(24,57){\colorbox{gray}}

\put(30,48){\colorbox{gray}}
\put(30,52){\colorbox{gray}}
\put(30,57){\colorbox{gray}}
\put(34,48){\colorbox{gray}}
\put(34,52){\colorbox{gray}}
\put(34,57){\colorbox{gray}}
\put(39,48){\colorbox{gray}}
\put(39,52){\colorbox{gray}}
\put(39,57){\colorbox{gray}}

\put(45,48){\colorbox{gray}}
\put(45,52){\colorbox{gray}}
\put(45,57){\colorbox{gray}}
\put(49,48){\colorbox{gray}}
\put(49,52){\colorbox{gray}}
\put(49,57){\colorbox{gray}}
\put(54,48){\colorbox{gray}}
\put(54,52){\colorbox{gray}}
\put(54,57){\colorbox{gray}}

\put(60,48){\colorbox{gray}}
\put(60,52){\colorbox{gray}}
\put(60,57){\colorbox{gray}}
\put(64,48){\colorbox{gray}}
\put(64,52){\colorbox{gray}}
\put(64,57){\colorbox{gray}}
\put(69,48){\colorbox{gray}}
\put(69,52){\colorbox{gray}}
\put(69,57){\colorbox{gray}}

\put(0,33){\colorbox{gray}}
\put(0,37){\colorbox{gray}}
\put(0,42){\colorbox{gray}}
\put(4,33){\colorbox{gray}}
\put(4,37){\colorbox{gray}}
\put(4,42){\colorbox{gray}}
\put(9,33){\colorbox{gray}}
\put(9,37){\colorbox{gray}}
\put(9,42){\colorbox{gray}}

\put(15,33){\colorbox{gray}}
\put(15,37){\colorbox{gray}}
\put(15,42){\colorbox{gray}}
\put(19,33){\colorbox{gray}}
\put(19,37){\colorbox{gray}}
\put(19,42){\colorbox{gray}}
\put(24,33){\colorbox{gray}}
\put(24,37){\colorbox{gray}}
\put(24,42){\colorbox{gray}}

\put(30,33){\colorbox{gray}}
\put(30,37){\colorbox{gray}}
\put(30,42){\colorbox{gray}}
\put(34,33){\colorbox{gray}}
\put(34,37){\colorbox{gray}}
\put(34,42){\colorbox{gray}}
\put(39,33){\colorbox{gray}}
\put(39,37){\colorbox{gray}}
\put(39,42){\colorbox{gray}}

\put(45,33){\colorbox{gray}}
\put(45,37){\colorbox{gray}}
\put(45,42){\colorbox{gray}}
\put(49,33){\colorbox{gray}}
\put(49,37){\colorbox{gray}}
\put(49,42){\colorbox{gray}}
\put(54,33){\colorbox{gray}}
\put(54,37){\colorbox{gray}}
\put(54,42){\colorbox{gray}}

\put(60,33){\colorbox{gray}}
\put(60,37){\colorbox{gray}}
\put(60,42){\colorbox{gray}}
\put(64,33){\colorbox{gray}}
\put(64,37){\colorbox{gray}}
\put(64,42){\colorbox{gray}}
\put(69,33){\colorbox{gray}}
\put(69,37){\colorbox{gray}}
\put(69,42){\colorbox{gray}}

\put(30,18){\colorbox{gray}}
\put(30,22){\colorbox{gray}}
\put(30,27){\colorbox{gray}}
\put(34,18){\colorbox{gray}}
\put(34,22){\colorbox{gray}}
\put(34,27){\colorbox{gray}}
\put(39,18){\colorbox{gray}}
\put(39,22){\colorbox{gray}}
\put(39,27){\colorbox{gray}}

\put(45,18){\colorbox{gray}}
\put(45,22){\colorbox{gray}}
\put(45,27){\colorbox{gray}}
\put(49,18){\colorbox{gray}}
\put(49,22){\colorbox{gray}}
\put(49,27){\colorbox{gray}}
\put(54,18){\colorbox{gray}}
\put(54,22){\colorbox{gray}}
\put(54,27){\colorbox{gray}}

\put(60,18){\colorbox{gray}}
\put(60,22){\colorbox{gray}}
\put(60,27){\colorbox{gray}}
\put(64,18){\colorbox{gray}}
\put(64,22){\colorbox{gray}}
\put(64,27){\colorbox{gray}}
\put(69,18){\colorbox{gray}}
\put(69,22){\colorbox{gray}}
\put(69,27){\colorbox{gray}}

\put(45,3){\colorbox{gray}}
\put(45,7){\colorbox{gray}}
\put(45,12){\colorbox{gray}}
\put(49,3){\colorbox{gray}}
\put(49,7){\colorbox{gray}}
\put(49,12){\colorbox{gray}}
\put(54,3){\colorbox{gray}}
\put(54,7){\colorbox{gray}}
\put(54,12){\colorbox{gray}}

\put(60,3){\colorbox{gray}}
\put(60,7){\colorbox{gray}}
\put(60,12){\colorbox{gray}}
\put(64,3){\colorbox{gray}}
\put(64,7){\colorbox{gray}}
\put(64,12){\colorbox{gray}}
\put(69,3){\colorbox{gray}}
\put(69,7){\colorbox{gray}}
\put(69,12){\colorbox{gray}}

\put(0,0){\framebox(15,15)}
\put(15,0){\framebox(15,15)}
\put(30,0){\framebox(15,15)}
\put(45,0){\framebox(15,15)}
\put(60,0){\framebox(15,15)}
\put(0,15){\framebox(15,15)}
\put(15,15){\framebox(15,15)}
\put(30,15){\framebox(15,15)}
\put(45,15){\framebox(15,15)}
\put(60,15){\framebox(15,15)}
\put(0,30){\framebox(15,15)}
\put(15,30){\framebox(15,15)}
\put(30,30){\framebox(15,15)}
\put(45,30){\framebox(15,15)}
\put(60,30){\framebox(15,15)}
\put(0,45){\framebox(15,15)}
\put(15,45){\framebox(15,15)}
\put(30,45){\framebox(15,15)}
\put(45,45){\framebox(15,15)}
\put(60,45){\framebox(15,15)}
\put(0,60){\framebox(15,15)}
\put(15,60){\framebox(15,15)}
\put(30,60){\framebox(15,15)}
\put(45,60){\framebox(15,15)}
\put(60,60){\framebox(15,15)}
\end{picture}
\end{center}
\caption{Colored boxes for $h=(3,3,4,5,5).$ }
\label{picture:h=(3,3,4,5,5)}
\end{figure}
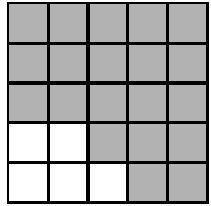 

\end{example}

Let $X: \C^n \to \C^n$ be a linear operator on $\C^n$. The (type A) \textbf{Hessenberg variety $\Hess(X,h)$ associated to the choice of $X$ and $h$} is defined as the following subvariety of $\Fl(\C^n)$:
\begin{equation}\label{eq: def Hess}
\Hess(X,h) :=\{V_{\bullet} \in \Fl(\C^n) \mid XV_i \subset V_{h(i)} \ {\rm for} \ i=1,2,\ldots,n \} \subseteq \Fl(\C^n). 
\end{equation}
In this manuscript we focus on the case when $X=S$ is a (choice of) regular semisimple operator, i.e., a diagonalizable matrix with distinct eigenvalues. In this situation we call $\Hess(S,h)$ a \textbf{regular semisimple Hessenberg variety}. From \cite[Theorem 6]{DeMa-Pr-Sh} we know that $\Hess(S,h)$ is smooth and equidimensional, and $\dim_{\C}(\Hess(S,h)) = \sum_{j=1}^n (h(j)-j)$. 
From the definition~\eqref{eq: def Hess} it is straightforward to see that when $h=(n,n,\ldots,n)$ is the maximal Hessenberg function, then the conditions given in~\eqref{eq: def Hess} are automatically satisfied and thus $\Hess(S,(n,n,\ldots,n)) = \Fl(\C^n)$. On the other hand, when $h=(2,3,4,\ldots,n,n)$ is the minimal Hessenberg function satisfying $h(j) \geq j+1$ for $1\leq j \leq n-1$, then it is known that $\Hess(S,h)$ is the toric variety associated to the fan of the type A Weyl chambers \cite[Theorem 11]{DeMa-Pr-Sh}. This regular semisimple Hessenberg variety is also known as the \textbf{permutohedral variety}. The permutohedral variety is a central object of interest in Sections~\ref{sec:decomposing} and~\ref{section: Richardson}.

Now we recall the definition of Schubert and opposite Schubert varieties and some of their basic properties. Let $G = \GL_n(\C)$ the general linear group and let $B$ and $B^{-}$ be the subgroups of upper-triangular and lower-triangular matrices in $G$ respectively. 
It is well-known that the flag variety $\Fl(\C^n)$ can be identified with the homogeneous space $G/B$. 
Let $\mathfrak{S}_n$ denote the permutation group on $n$ letters. For an element $w \in \mathfrak{S}_n$ we denote its one-line notation by either 
\[
w(1) \, w(2) \, \cdots \, w(n) \quad \textup{ or } \quad [w(1), w(2), \cdots, w(n)]. 
\]
For $w \in \mathfrak{S}_n$, we let $X_w$ denote the {\bf Schubert variety} associated to $w$, defined to be the closure of the $B$-orbit of the permutation flag $wB/B$ and let $X^w$ denote the \textbf{opposite Schubert variety}, defined to be the closure of the $B^-$-orbit of $wB/B$. Since $B^- = w_0 B w_0$ where $w_0$ is the longest element in $\mathfrak{S}_n$, we have the relation $X^w = w_0(X_{w_0 w})$. 
It is well-known that any opposite Schubert variety $X^w$ is irreducible and the codimension of $X^w$ in the flag variety $\Fl(\C^n)$ is the length $\ell(w)$ of $w$. Similarly $X_w$ is irreducible and its dimension is $\ell(w)$ for any $w\in \mathfrak{S}_n$. 
The Poincar\'e dual of the Schubert variety $X^w$ considered as a cycle of $\Fl(\C^n)$, which we denote by $[X^w] \in H^{2\ell(w)}(\Fl(\C^n))$, is called the {\bf Schubert class} corresponding to $w$. 
It is known that $[X^w] = [w_0(X_{w_0 w})] = [X_{w_0 w}]$, and also that the Schubert classes form an additive basis of the cohomology $H^*(\Fl(\C^n))$. (This result is also valid for the cohomology with coefficients in $\Z$.)

We are now in a position to state a result of Anderson and Tymoczko \cite{AT} which gives a formula for the cohomology class (i.e. Poincar\'e dual) of $\Hess(S,h)$ in terms of those of Schubert varieties. To state the precise result, we first need to define a certain permutation associated to each Hessenberg function $h$ as follows. As the base case we define 
\[
w_h(1) := n- h(1)+1. 
\]
We then inductively define 
\begin{align} \label{eq:wh}
\wh(i)&=\text{$(n-h(i)+1)$-th number of a set $[n]\setminus\{\wh(1),\ldots,\wh(i-1)\}$}. 
\end{align}

Anderson and Tymoczko show that the Poincar\'e dual of a regular semisimple Hessenberg variety $\Hess(S,h)$ considered as a cycle in the flag variety $\Fl(\C^n)$, which we denote as $[\Hess(S,h)]$, can be written in terms of opposite Schubert classes $[X^w]$ as follows.

\begin{theorem} (\cite[Corollary~3.3 (a) and equation~(14)]{AT}) \label{theorem:AT}
Let $\Hess(S,h)$ be a regular semisimple Hessenberg variety and let $\{X^w \}_{w \in \mathfrak{S}_n}$ denote the opposite Schubert varieties.
Then the class $[\Hess(S,h)]$ can be expressed as
\begin{equation}\label{eq: AT formula}
[\Hess(S,h)]=\sum_{\substack{u,v \in \mathfrak{S}_n \\ v^{-1}u=\wh \\ \ell(u)+\ell(v)=\ell(\wh)}}[X^u][X^{w_0vw_0}] \in H^*(\Fl(\C^n)),
\end{equation}
where $w_0$ is the longest element in $\mathfrak{S}_n$ and $\wh$ is the permutation defined above. 
\end{theorem}

We illustrate the above theorem in the case of some small-$n$ permutohedral varieties.

\begin{example} \label{exam:n=3,4Permutohedral variety}
Let $h=(2,3,4,\ldots, n, n)$. This corresponds to the case of the permutohedral variety mentioned above. In this case, it is straightforward to see that the permutation $w_h$ defined above is the longest element in the subgroup $\mathfrak{S}_{n-1}$ of $\mathfrak{S}_n$, i.e., 
\[
w_h = [ n-1 , \, n-2 \, , \cdots \, 2 , \, 1 ,\, n] 
\]
in one-line notation. We now compute the cohomology classes $[\Hess(S,h)]$ in the cases $n=3$ and $n=4$. 
\begin{enumerate} 
\item[(a)] Suppose $n=3$. Then $h=(2,3,3)$ and $w_h = [ 2, 1, 3 ] = s_1$. According to Theorem~\ref{theorem:AT} we need to find pairs $v,u \in \mathfrak{S}_3$ such that $v^{-1}u = s_1$ and $\ell(u)+\ell(v) = \ell(w_h)=\ell(s_1)=1$, but there are only two such choices: $v=s_1$ and $u=\mathrm{id}$, or, $v=\mathrm{id}$ and $u=s_1$. Therefore, by Theorem~\ref{theorem:AT} we conclude, in this case, that 
\[
[\Hess(S,h)] =  [ X^{s_1}][X^{\mathrm{id}}] + [X^{\mathrm{id}}] [X^{s_2}] = [X^{213}] + [X^{132}].
\]
\item[(b)] Now suppose $n=4$. Consider $h=(2,3,4,4)$. It is straightforward to compute $w_h=[ 3, 2, 1, 4] = s_1s_2s_1=s_2s_1s_2$. Using Theorem~\ref{theorem:AT} again we have 
\begin{align*}
[\Hess(S,h)]=&[X^{\id}][X^{w_0(s_1s_2s_1)^{-1}w_0}]+[X^{s_1}][X^{w_0(s_1s_2)^{-1}w_0}]+[X^{s_2}][X^{w_0(s_2s_1)^{-1}w_0}]\\
&+[X^{s_2s_1}][X^{w_0(s_1)^{-1}w_0}]+[X^{s_1s_2}][X^{w_0(s_2)^{-1}w_0}]+[X^{s_1s_2s_1}][X^{\id}]\\
=&[X^{1432}]+[X^{s_1}][X^{1342}]+[X^{s_2}][X^{1423}]+[X^{3124}][X^{s_3}]+[X^{2314}][X^{s_2}]+[X^{3214}]. 
\end{align*}
Using Monk's formula \cite{Monk} (cf. also \cite[p.180]{fult97}) we obtain 
\begin{align*}
[X^{s_1}][X^{1342}]=&[X^{3142}]+[X^{2341}]  \\
[X^{s_2}][X^{1423}]=&[X^{2413}]  \\
[X^{s_3}][X^{3124}]=&[X^{4123}]+[X^{3142}]  \\
[X^{s_2}][X^{2314}]=&[X^{2413}]  
\end{align*}
which therefore yields the simplified formula 
\begin{equation*}
[\Hess(S,h)]=[X^{1432}]+[X^{2341}]+2[X^{2413}]+2[X^{3142}]+[X^{3214}]+[X^{4123}]. 
\end{equation*}
\end{enumerate} 

\end{example}

Next, we recall the fundamental combinatorial objects considered in this manuscript, namely, the Gelfand-Zetlin polytope and its faces. 
Let $\lambda=(\lambda_1,\lambda_2,\cdots,\lambda_n)\in\R^n$ with $\lambda_1>\lambda_2>\cdots>\lambda_n$.
The \textbf{Gelfand--Zetlin polytope $\GZ(\lambda)$} in $\R^{\m}$, where $\m = n(n-1)/2$, consists of the set of points $(x_{i,j})_{i,j} \in \R^{\m}$ satisfying the inequalities defined by the following diagram: 
\begin{equation}\label{eq:GZ}
    \begin{array}{cccccc}
        \lambda_1&  x_{1,2}    &    x_{1,3}   &     \cdots    & x_{1,n-1}	&x_{1,n}\\
                         &\lambda_2&  x_{2,3}       &     \ddots  &x_{2,n-1} 	&x_{2,n} \\
                         &                 &  \ddots	      &		\ddots    &\vdots	      &\vdots\\
                         &                 &                  &\lambda_{n-2}&x_{n-2,n-1}	 &x_{n-2,n}\\
                         &                 &                  &                      &\lambda_{n-1}&x_{n-1,n}\\
                         &                 &                  &                      &                      &\lambda_n
    \end{array}
\end{equation}
where any set of three variables found in a triangular arrangement as follows 
\begin{equation} \label{eq:abc}
    \begin{array}{cc}
        a& b\\
          &c
    \end{array}
\end{equation}
in the diagram~\eqref{eq:GZ} must satisfy $a \geq b \geq c$.
Here, we set $x_{j,j}:=\lambda_j$ for all $j=1,2,\ldots,n$.

We will represent faces of $\GZ(\lambda)$ by \textbf{face diagrams}. This representation was first introduced by~\cite{Ko2000} and we will use the modified version given in~\cite{KST12}, which we now review. Recall that a face of $\GZ(\lambda)$ is specified by a collection of equations, each of which is of the form either $a=b$ or $b=c$, where $a,b,c$ are three variables appearing in~\eqref{eq:GZ} in a triangular arrangement as in~\eqref{eq:abc}.
With this in mind, we define face diagrams as follows. 
Firstly, we replace each symbol (either $\lambda_j=x_{j,j}$ or $x_{i,j}$) appearing  in~\eqref{eq:GZ} by a dot. Secondly, given an equation $a=b$ (respectively $b=c$) as above, we graphically represent this $a=b$ by drawing a line segment connecting the corresponding dots; note that since $a$ and $b$ are adjacent in a row (respectively a column), this line segment is horizontal (respectively vertical), or in other words, ``east-west'' (respectively ``north-south''). A system of such equations, defining a face of $\GZ(\lambda)$, is hence represented by a collection of east-west and north-south line segments. We call this collection the \textbf{face diagram} (of that face). See Figure~\ref{fig:ex:face_diagram} for examples.

\begin{figure}[h]
\begin{center}
\begin{subfigure}[b]{.24\textwidth}
\begin{center}
    \begin{tikzpicture}[scale=0.6,font=\footnotesize]
    \tikzset{
    solid node/.style={circle,draw,inner sep=1.2,fill=black}
    }
        \node [solid node](01) at (0,3) {};
        \node [solid node](02) at (1,2) {};
        \node [solid node](03) at (2,1) {};
        \node [solid node](04) at (3,0) {};
        \node [solid node](11) at (1,3) {};
        \node [solid node](12) at (2,2) {};
        \node [solid node](13) at (3,1) {};
        \node [solid node](21) at (2,3) {};
        \node [solid node](22) at (3,2) {};
        \node [solid node](31) at (3,3) {};
        \path[thick] (01) edge (11);
    \end{tikzpicture}
    \end{center}
        \subcaption*{$\lambda_1=x_{1,2}$}
    \end{subfigure}
    \begin{subfigure}[b]{.24\textwidth}
        \begin{center}
    \begin{tikzpicture}[scale=0.6,font=\footnotesize]
    \tikzset{
    solid node/.style={circle,draw,inner sep=1.2,fill=black}
    }
        \node [solid node](01) at (0,3) {};
        \node [solid node](02) at (1,2) {};
        \node [solid node](03) at (2,1) {};
        \node [solid node](04) at (3,0) {};
        \node [solid node](11) at (1,3) {};
        \node [solid node](12) at (2,2) {};
        \node [solid node](13) at (3,1) {};
        \node [solid node](21) at (2,3) {};
        \node [solid node](22) at (3,2) {};
        \node [solid node](31) at (3,3) {};
        \path[thick] (04) edge (13);
        \end{tikzpicture}
        \end{center}
        \subcaption*{$x_{3,4}=\lambda_4$}
    \end{subfigure}
    \begin{subfigure}[b]{.37\textwidth}
    \begin{center}
    \begin{tikzpicture}[scale=0.6,font=\footnotesize]
        \tikzset{
        solid node/.style={circle,draw,inner sep=1.2,fill=black},
        hollow node/.style={circle,draw,inner sep=1.2}
        }
        \node [solid node](01) at (0,3) {};
        \node [solid node](02) at (1,2) {};
        \node [solid node](03) at (2,1) {};
        \node [solid node](04) at (3,0) {};
        \node [solid node](11) at (1,3) {};
        \node [solid node](12) at (2,2) {};
        \node [solid node](13) at (3,1) {};
        \node [solid node](21) at (2,3) {};
        \node [solid node](22) at (3,2) {};
        \node [solid node](31) at (3,3) {};
        \path (11) edge (01);
        \path (13) edge (04);
        \path (21) edge (11);
    \end{tikzpicture}
\end{center}
\subcaption*{$\lambda_1=x_{1,2}=x_{1,3}$; $x_{3,4}=\lambda_4$}
\end{subfigure}
    \end{center}
    \caption{Examples of face diagrams when $n$=4}\label{fig:ex:face_diagram}
\end{figure}
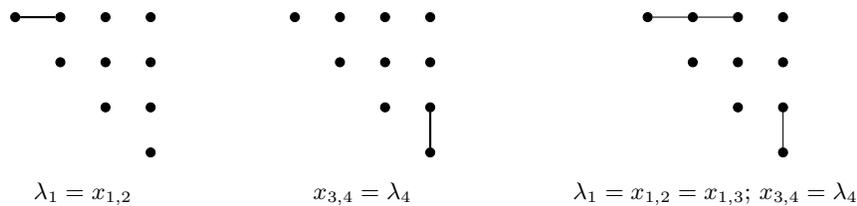

It may be helpful to the reader to explicitly visualize the facets on the polytope itself. In the example below we illustrate all $6$ facets of the $n=3$ Gelfand-Zetlin polytope as both a face diagram and as a 
$2$-dimensional facet on the $3$-dimensional Gelfand-Zetlin polytope.

\begin{example}\label{example: volume for n=3 h=233}
Let $n=3$ and $\lambda_1 > \lambda_2 > \lambda_3$. In this small case we use notation $x=x_{1,2}$, $y=x_{2,3}$ and $z=x_{1,3}$, and the polytope $\GZ(\lambda) \subseteq \R^3$ is given by the set of $(x,y,z) \in \R^3$ satisfying the inequalities associated to the diagram 
\begin{equation}\label{eq:GZ n=3}
    \begin{array}{ccc}
        \lambda_1&  x    &    z \\   
         & \lambda_2 & y \\
          & & \lambda_3.
     \end{array}
\end{equation}
In Figure~\ref{fig: facets and dots for n=3} we have drawn both the dot diagram and a shading of the corresponding facet on the Gelfand-Zetlin polytope for all $6$ facets of $\GZ(\lambda)$.

\begin{figure}[h]
\begin{center}
\begin{subfigure}[b]{.36\textwidth}
\centering
\begin{tikzpicture}[scale=.5]
	\fill[yellow] (3.5,2.3)--(3.5,-2)--(2,0)--(2,2)--cycle;
	\draw (0,0)--(1.5,-2)--(3.5,-2)--(3.5,2.3)--(2,2)--cycle;
	\draw (0,0)--(2,0)--(2,2);
	\draw (2,0)--(3.5,-2);
	\draw[dashed] (0,0)--(1.5,0.3)--(3.5,2.3);
	\draw[dashed] (1.5,0.3)--(1.5,-2);
	
	\tikzset{
        solid node/.style={circle,draw,inner sep=1.2,fill=black},
        hollow node/.style={circle,draw,inner sep=1.2}
        }
        \node [solid node](01) at (0+5,3-2) {};
        \node [solid node](02) at (1+5,2-2) {};
        \node [solid node](03) at (2+5,1-2) {};
        \node [solid node](11) at (1+5,3-2) {};
        \node [solid node](12) at (2+5,2-2) {};
        \node [solid node](21) at (2+5,3-2) {};
        \path (11) edge (01);
\end{tikzpicture}
\caption*{$\lambda_1=x$}
\end{subfigure}
\begin{subfigure}[b]{.36\textwidth}
\centering
\begin{tikzpicture}[scale=.5]
	\fill[yellow] (0,0)--(1.5,-2)--(1.5,0.3)--cycle;
	\draw (0,0)--(1.5,-2)--(3.5,-2)--(3.5,2.3)--(2,2)--cycle;
	\draw (0,0)--(2,0)--(2,2);
	\draw (2,0)--(3.5,-2);
	\draw[dashed] (0,0)--(1.5,0.3)--(3.5,2.3);
	\draw[dashed] (1.5,0.3)--(1.5,-2);
	
	\tikzset{
        solid node/.style={circle,draw,inner sep=1.2,fill=black},
        hollow node/.style={circle,draw,inner sep=1.2}
        }
        \node [solid node](01) at (0+5,3-2) {};
        \node [solid node](02) at (1+5,2-2) {};
        \node [solid node](03) at (2+5,1-2) {};
        \node [solid node](11) at (1+5,3-2) {};
        \node [solid node](12) at (2+5,2-2) {};
        \node [solid node](21) at (2+5,3-2) {};
        \path (11) edge (02);
\end{tikzpicture}
\caption*{$x=\lambda_2$}
\end{subfigure}
\begin{subfigure}[b]{.36\textwidth}
\centering
\begin{tikzpicture}[scale=.5]
	\fill[yellow] (0,0)--(2,2)--(2,0)--cycle;
	\draw (0,0)--(1.5,-2)--(3.5,-2)--(3.5,2.3)--(2,2)--cycle;
	\draw (0,0)--(2,0)--(2,2);
	\draw (2,0)--(3.5,-2);
	\draw[dashed] (0,0)--(1.5,0.3)--(3.5,2.3);
	\draw[dashed] (1.5,0.3)--(1.5,-2);
	
	\tikzset{
        solid node/.style={circle,draw,inner sep=1.2,fill=black},
        hollow node/.style={circle,draw,inner sep=1.2}
        }
        \node [solid node](01) at (0+5,3-2) {};
        \node [solid node](02) at (1+5,2-2) {};
        \node [solid node](03) at (2+5,1-2) {};
        \node [solid node](11) at (1+5,3-2) {};
        \node [solid node](12) at (2+5,2-2) {};
        \node [solid node](21) at (2+5,3-2) {};
        \path (02) edge (12);
\end{tikzpicture}
\caption*{$\lambda_2=y$}
\end{subfigure}
\begin{subfigure}[b]{.36\textwidth}
\centering
\begin{tikzpicture}[scale=.5]
	\fill[yellow] (1.5,0.3)--(3.5,2.3)--(3.5,-2)--(1.5,-2)--cycle;
	\draw (0,0)--(1.5,-2)--(3.5,-2)--(3.5,2.3)--(2,2)--cycle;
	\draw (0,0)--(2,0)--(2,2);
	\draw (2,0)--(3.5,-2);
	\draw[dashed] (0,0)--(1.5,0.3)--(3.5,2.3);
	\draw[dashed] (1.5,0.3)--(1.5,-2);
	
	\tikzset{
        solid node/.style={circle,draw,inner sep=1.2,fill=black},
        hollow node/.style={circle,draw,inner sep=1.2}
        }
        \node [solid node](01) at (0+5,3-2) {};
        \node [solid node](02) at (1+5,2-2) {};
        \node [solid node](03) at (2+5,1-2) {};
        \node [solid node](11) at (1+5,3-2) {};
        \node [solid node](12) at (2+5,2-2) {};
        \node [solid node](21) at (2+5,3-2) {};
        \path (03) edge (12);
\end{tikzpicture}
\caption*{$y=\lambda_3$}
\end{subfigure}
\begin{subfigure}[b]{.36\textwidth}
\centering
\begin{tikzpicture}[scale=.5]
	\fill[yellow] (0,0)--(1.5,0.3)--(3.5,2.3)--(2,2)--cycle;
	\draw (0,0)--(1.5,-2)--(3.5,-2)--(3.5,2.3)--(2,2)--cycle;
	\draw (0,0)--(2,0)--(2,2);
	\draw (2,0)--(3.5,-2);
	\draw[dashed] (0,0)--(1.5,0.3)--(3.5,2.3);
	\draw[dashed] (1.5,0.3)--(1.5,-2);
	
	\tikzset{
        solid node/.style={circle,draw,inner sep=1.2,fill=black},
        hollow node/.style={circle,draw,inner sep=1.2}
        }
        \node [solid node](01) at (0+5,3-2) {};
        \node [solid node](02) at (1+5,2-2) {};
        \node [solid node](03) at (2+5,1-2) {};
        \node [solid node](11) at (1+5,3-2) {};
        \node [solid node](12) at (2+5,2-2) {};
        \node [solid node](21) at (2+5,3-2) {};
        \path (11) edge (21);
\end{tikzpicture}
\caption*{$x=z$}
\end{subfigure}
\begin{subfigure}[b]{.36\textwidth}
\centering
\begin{tikzpicture}[scale=.5]
	\fill[yellow] (0,0)--(2,0)--(3.5,-2)--(1.5,-2)--cycle;
	\draw (0,0)--(1.5,-2)--(3.5,-2)--(3.5,2.3)--(2,2)--cycle;
	\draw (0,0)--(2,0)--(2,2);
	\draw (2,0)--(3.5,-2);
	\draw[dashed] (0,0)--(1.5,0.3)--(3.5,2.3);
	\draw[dashed] (1.5,0.3)--(1.5,-2);
	
	\tikzset{
        solid node/.style={circle,draw,inner sep=1.2,fill=black},
        hollow node/.style={circle,draw,inner sep=1.2}
        }
        \node [solid node](01) at (0+5,3-2) {};
        \node [solid node](02) at (1+5,2-2) {};
        \node [solid node](03) at (2+5,1-2) {};
        \node [solid node](11) at (1+5,3-2) {};
        \node [solid node](12) at (2+5,2-2) {};
        \node [solid node](21) at (2+5,3-2) {};
        \path (21) edge (12);
\end{tikzpicture}
\caption*{$z=y$}
\end{subfigure}
\end{center}
\caption{Facets of $\GZ(\lambda)$ for $n=3$ with $\lambda_1 > \lambda_2 > \lambda_3$.}
\label{fig: facets and dots for n=3}
\end{figure}
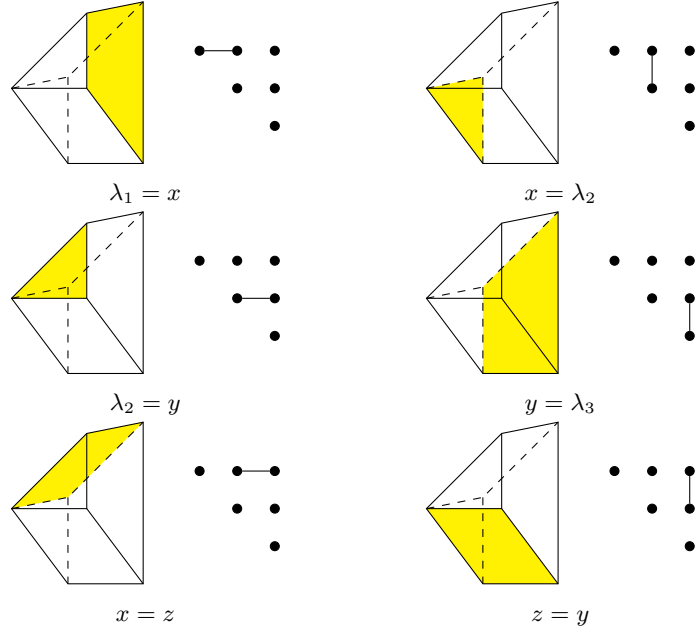

\end{example}

With the above notation in place, we can now understand the RHS of~\eqref{eq: AT formula} 
in terms of certain faces of the Gelfand-Zetlin polytope. The discussion below is a summary of results of Kiritchenko-Smirnov-Timorin \cite{KST12}; 
we refer the reader to \cite[Section~3.3, 4.3]{KST12} for more details.  A \textbf{Kogan face} is a face of $\GZ(\lambda)$ obtained via equations of the form $x_{i,j} = x_{i,j+1}$. In the language of face diagrams, these are the east-west line segments. A {\bf dual Kogan face} is a face of the Gelfand--Zetlin polytope given by equations of the type $x_{i,j}=x_{i+1,j}$. These are the north-south line segments in the face diagram. 
To each Kogan face $F$ we associate to it a permutation $w(F)$ as follows. To each 
line segment connecting $(i,j)$ and $(i, j+1)$ in the face diagram of $F$, we assign the simple transposition $s_{j} = (j, j+1)$. To construct $w(F)$, we now successively compose those transpositions corresponding to those line segments by reading the transpositions as follows: we read along the diagonals, going (diagonally) down along each diagonal, and starting from the outermost (shortest) diagonal and ending at the main (longest) diagonal. The resulting permutation is denoted $w(F)$. See Figure~\ref{fig: wF example} for an example. We say that a Kogan face is \textbf{reduced} if the decomposition for $w(F)$ obtained by the procedure above is reduced.

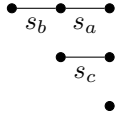
\begin{figure}[h]
\begin{center}
\begin{subfigure}[b]{.24\textwidth}
\centering
\begin{tikzpicture}[scale=.65]
\tikzset{
        solid node/.style={circle,draw,inner sep=1.2,fill=black},
        hollow node/.style={circle,draw,inner sep=1.2}
        }
        \node [solid node](01) at (0,3) {};
        \node [solid node](02) at (1,2) {};
        \node [solid node](03) at (2,1) {};
        \node [solid node](11) at (1,3) {};
        \node [solid node](12) at (2,2) {};
        \node [solid node](21) at (2,3) {};
        \path (11) edge node[below]{$s_b$} (01);
        \path (11) edge node[below]{$s_a$} (21);
        \path (02) edge node[below]{$s_c$} (12);
\end{tikzpicture}
\end{subfigure}
\end{center}
\caption{In the figure above, the corresponding permutation $w(F)$ is $w=s_as_bs_c$. We start by reading along the ``outermost'' (shortest) diagonal, reading (diagonally) downwards, and then proceed to the next diagonal on its left, again reading (diagonally) downwards.} 
\label{fig: wF example} 
\end{figure}

There is a similar procedure which we apply to each dual Kogan face $F^*$. To each line segment connecting $(i,j)$ and $(i+1,j)$ in the face diagram of $F^*$, we associate the simple transposition $s_{n-i}$. Now we 
successively compose those transpositions which appear in the face diagram of $F^*$ by reading them along the diagonals, going (diagonally) up along each diagonal, and starting from the outermost (shortest) diagonal and ending at the main (longest) diagonal. The resulting permutation is denoted $w(F^*)$. 
We say that a dual Kogan face $F^*$ is {\bf reduced} if the decomposition for $w(F^*)$ obtained in this way is reduced.

\begin{example}\label{example: kogan and dual kogan} 

Let $n=3$. We list all the reduced Kogan faces and reduced dual Kogan faces in Figures~\ref{fig: reduced Kogan} and~\ref{fig: reduced dual Kogan}. 

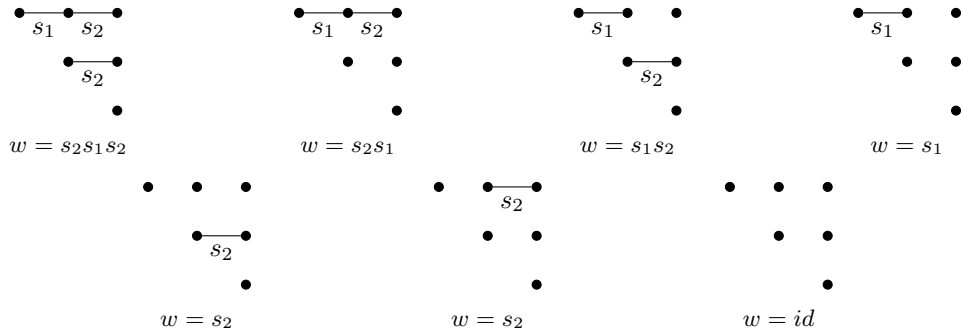
\begin{figure}[h]
\begin{center}
\begin{subfigure}[b]{.24\textwidth}
\centering
\begin{tikzpicture}[scale=.65]
\tikzset{
        solid node/.style={circle,draw,inner sep=1.2,fill=black},
        hollow node/.style={circle,draw,inner sep=1.2}
        }
        \node [solid node](01) at (0,3) {};
        \node [solid node](02) at (1,2) {};
        \node [solid node](03) at (2,1) {};
        \node [solid node](11) at (1,3) {};
        \node [solid node](12) at (2,2) {};
        \node [solid node](21) at (2,3) {};
        \path (11) edge node[below]{$s_1$} (01);
        \path (11) edge node[below]{$s_2$} (21);
        \path (02) edge node[below]{$s_2$} (12);
\end{tikzpicture}
\caption*{$w=s_2s_1s_2$}
\end{subfigure}
\begin{subfigure}[b]{.24\textwidth}
\centering
\begin{tikzpicture}[scale=.65]
\tikzset{
        solid node/.style={circle,draw,inner sep=1.2,fill=black},
        hollow node/.style={circle,draw,inner sep=1.2}
        }
        \node [solid node](01) at (0,3) {};
        \node [solid node](02) at (1,2) {};
        \node [solid node](03) at (2,1) {};
        \node [solid node](11) at (1,3) {};
        \node [solid node](12) at (2,2) {};
        \node [solid node](21) at (2,3) {};
        \path (11) edge node[below]{$s_1$} (01);
        \path (11) edge node[below]{$s_2$} (21);
\end{tikzpicture}
\caption*{$w=s_2s_1$}
\end{subfigure}
\begin{subfigure}[b]{.24\textwidth}
\centering
\begin{tikzpicture}[scale=.65]
\tikzset{
        solid node/.style={circle,draw,inner sep=1.2,fill=black},
        hollow node/.style={circle,draw,inner sep=1.2}
        }
        \node [solid node](01) at (0,3) {};
        \node [solid node](02) at (1,2) {};
        \node [solid node](03) at (2,1) {};
        \node [solid node](11) at (1,3) {};
        \node [solid node](12) at (2,2) {};
        \node [solid node](21) at (2,3) {};
        \path (11) edge node[below]{$s_1$} (01);
        \path (02) edge node[below]{$s_2$} (12);
\end{tikzpicture}
\caption*{$w=s_1s_2$}
\end{subfigure}
\begin{subfigure}[b]{.24\textwidth}
\centering
\begin{tikzpicture}[scale=.65]
\tikzset{
        solid node/.style={circle,draw,inner sep=1.2,fill=black},
        hollow node/.style={circle,draw,inner sep=1.2}
        }
        \node [solid node](01) at (0,3) {};
        \node [solid node](02) at (1,2) {};
        \node [solid node](03) at (2,1) {};
        \node [solid node](11) at (1,3) {};
        \node [solid node](12) at (2,2) {};
        \node [solid node](21) at (2,3) {};
        \path (11) edge node[below]{$s_1$} (01);\end{tikzpicture}
\caption*{$w=s_1$}
\end{subfigure}

\vspace{.25cm}

\begin{subfigure}[b]{.25\textwidth}
\centering
\begin{tikzpicture}[scale=.65]
\tikzset{
        solid node/.style={circle,draw,inner sep=1.2,fill=black},
        hollow node/.style={circle,draw,inner sep=1.2}
        }
        \node [solid node](01) at (0,3) {};
        \node [solid node](02) at (1,2) {};
        \node [solid node](03) at (2,1) {};
        \node [solid node](11) at (1,3) {};
        \node [solid node](12) at (2,2) {};
        \node [solid node](21) at (2,3) {};
        \path (02) edge node[below]{$s_2$} (12);
\end{tikzpicture}
\caption*{$w=s_2$}
\end{subfigure}
\begin{subfigure}[b]{.25\textwidth}
\centering
\begin{tikzpicture}[scale=.65]
\tikzset{
        solid node/.style={circle,draw,inner sep=1.2,fill=black},
        hollow node/.style={circle,draw,inner sep=1.2}
        }
        \node [solid node](01) at (0,3) {};
        \node [solid node](02) at (1,2) {};
        \node [solid node](03) at (2,1) {};
        \node [solid node](11) at (1,3) {};
        \node [solid node](12) at (2,2) {};
        \node [solid node](21) at (2,3) {};
        \path (11) edge node[below]{$s_2$} (21);
\end{tikzpicture}
\caption*{$w=s_2$}
\end{subfigure}
\begin{subfigure}[b]{.25\textwidth}
\centering
\begin{tikzpicture}[scale=.65]
\tikzset{
        solid node/.style={circle,draw,inner sep=1.2,fill=black},
        hollow node/.style={circle,draw,inner sep=1.2}
        }
        \node [solid node](01) at (0,3) {};
        \node [solid node](02) at (1,2) {};
        \node [solid node](03) at (2,1) {};
        \node [solid node](11) at (1,3) {};
        \node [solid node](12) at (2,2) {};
        \node [solid node](21) at (2,3) {};
\end{tikzpicture}
\caption*{$w=id$}
\end{subfigure}
\end{center}
\caption{Reduced Kogan faces when $n=3$}
\label{fig: reduced Kogan} 
\end{figure}

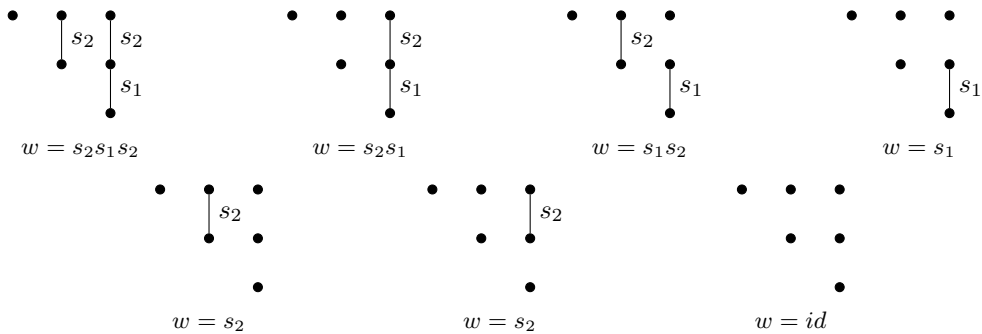
\begin{figure}[h]
\begin{center}
\begin{subfigure}[b]{.24\textwidth}
\centering
\begin{tikzpicture}[scale=.65]
\tikzset{
        solid node/.style={circle,draw,inner sep=1.2,fill=black},
        hollow node/.style={circle,draw,inner sep=1.2}
        }
        \node [solid node](01) at (0,3) {};
        \node [solid node](02) at (1,2) {};
        \node [solid node](03) at (2,1) {};
        \node [solid node](11) at (1,3) {};
        \node [solid node](12) at (2,2) {};
        \node [solid node](21) at (2,3) {};
        \path (03) edge node[right]{$s_1$} (12);
        \path (12) edge node[right]{$s_2$} (21);
        \path (02) edge node[right]{$s_2$} (11);
\end{tikzpicture}
\caption*{$w=s_2s_1s_2$}
\end{subfigure}
\begin{subfigure}[b]{.24\textwidth}
\centering
\begin{tikzpicture}[scale=.65]
\tikzset{
        solid node/.style={circle,draw,inner sep=1.2,fill=black},
        hollow node/.style={circle,draw,inner sep=1.2}
        }
        \node [solid node](01) at (0,3) {};
        \node [solid node](02) at (1,2) {};
        \node [solid node](03) at (2,1) {};
        \node [solid node](11) at (1,3) {};
        \node [solid node](12) at (2,2) {};
        \node [solid node](21) at (2,3) {};
        \path (03) edge node[right]{$s_1$} (12);
        \path (12) edge node[right]{$s_2$} (21);
\end{tikzpicture}
\caption*{$w=s_2s_1$}
\end{subfigure}
\begin{subfigure}[b]{.24\textwidth}
\centering
\begin{tikzpicture}[scale=.65]
\tikzset{
        solid node/.style={circle,draw,inner sep=1.2,fill=black},
        hollow node/.style={circle,draw,inner sep=1.2}
        }
        \node [solid node](01) at (0,3) {};
        \node [solid node](02) at (1,2) {};
        \node [solid node](03) at (2,1) {};
        \node [solid node](11) at (1,3) {};
        \node [solid node](12) at (2,2) {};
        \node [solid node](21) at (2,3) {};
        \path (03) edge node[right]{$s_1$} (12);
        \path (02) edge node[right]{$s_2$} (11);
\end{tikzpicture}
\caption*{$w=s_1s_2$}
\end{subfigure}
\begin{subfigure}[b]{.24\textwidth}
\centering
\begin{tikzpicture}[scale=.65]
\tikzset{
        solid node/.style={circle,draw,inner sep=1.2,fill=black},
        hollow node/.style={circle,draw,inner sep=1.2}
        }
        \node [solid node](01) at (0,3) {};
        \node [solid node](02) at (1,2) {};
        \node [solid node](03) at (2,1) {};
        \node [solid node](11) at (1,3) {};
        \node [solid node](12) at (2,2) {};
        \node [solid node](21) at (2,3) {};
        \path (03) edge node[right]{$s_1$} (12);\end{tikzpicture}
\caption*{$w=s_1$}
\end{subfigure}

\vspace{.25cm}

\begin{subfigure}[b]{.25\textwidth}
\centering
\begin{tikzpicture}[scale=.65]
\tikzset{
        solid node/.style={circle,draw,inner sep=1.2,fill=black},
        hollow node/.style={circle,draw,inner sep=1.2}
        }
        \node [solid node](01) at (0,3) {};
        \node [solid node](02) at (1,2) {};
        \node [solid node](03) at (2,1) {};
        \node [solid node](11) at (1,3) {};
        \node [solid node](12) at (2,2) {};
        \node [solid node](21) at (2,3) {};
        \path (02) edge node[right]{$s_2$} (11);
\end{tikzpicture}
\caption*{$w=s_2$}
\end{subfigure}
\begin{subfigure}[b]{.25\textwidth}
\centering
\begin{tikzpicture}[scale=.65]
\tikzset{
        solid node/.style={circle,draw,inner sep=1.2,fill=black},
        hollow node/.style={circle,draw,inner sep=1.2}
        }
        \node [solid node](01) at (0,3) {};
        \node [solid node](02) at (1,2) {};
        \node [solid node](03) at (2,1) {};
        \node [solid node](11) at (1,3) {};
        \node [solid node](12) at (2,2) {};
        \node [solid node](21) at (2,3) {};
        \path (12) edge node[right]{$s_2$} (21);
\end{tikzpicture}
\caption*{$w=s_2$}
\end{subfigure}
\begin{subfigure}[b]{.25\textwidth}
\centering
\begin{tikzpicture}[scale=.65]
\tikzset{
        solid node/.style={circle,draw,inner sep=1.2,fill=black},
        hollow node/.style={circle,draw,inner sep=1.2}
        }
        \node [solid node](01) at (0,3) {};
        \node [solid node](02) at (1,2) {};
        \node [solid node](03) at (2,1) {};
        \node [solid node](11) at (1,3) {};
        \node [solid node](12) at (2,2) {};
        \node [solid node](21) at (2,3) {};
\end{tikzpicture}
\caption*{$w=id$}
\end{subfigure}
\end{center}
\caption{Reduced dual Kogan faces when $n=3$}
\label{fig: reduced dual Kogan} 
\end{figure}

\end{example}

In what follows, we will use a special case of \cite[Corollary~4.6]{KST12}, stated below. Since we do not require the full details of this theorem, we do not recall precise definitions of all the notation that is used. Briefly, the expressions $[F \cap F^*]$ appearing below are elements of a polytope ring corresponding to a certain resolution of the Gelfand-Zetlin polytope, and the map $\pi$ identifies the sum shown below with an element in the polytope ring corresponding to $\GZ(\lambda)$, which is known to be isomorphic to the Chow (cohomology) ring of $\Fl(\C^n)$. For details, see \cite{KST12}.

\begin{theorem}\label{theorem: KST} (\cite[Corollary~4.6]{KST12}) 
Let the notation be as above. Then 
\begin{equation*}
[X^u][X^{w_0vw_0}]= \pi \, \bigg( \sum_{\substack{F:\text{ reduced Kogan face} \\ F^*:\text{ reduced dual Kogan face} \\ w(F)=u,\ w(F^*)=v}} [F\cap F^*]\bigg). 
\end{equation*}
\end{theorem}

Putting Theorems~\ref{theorem:AT} and~\ref{theorem: KST} together we immediately obtain the following explicit expression of the Poincar\'e dual $[\Hess(S,h)]$ of a regular semisimple Hessenberg variety as an expression involving a sum of certain faces of the Gelfand-Zetlin polytope.

\begin{corollary}\label{cor: Hess AT KST} 
Following the notation above, we have

\begin{equation} \label{eq:volHess(S,h)}
[\Hess(S,h)] = \pi \, \bigg( \sum_{\substack{u,v \in \mathfrak{S}_n \\ v^{-1}u=\wh \\ \ell(u)+\ell(v)=\ell(\wh)}} \sum_{\substack{F:\text{ reduced Kogan face} \\ F^*:\text{ reduced dual Kogan face} \\ w(F)=u,\ w(F^*)=v}}  [F\cap F^*] \bigg). 
\end{equation} 
\end{corollary}

Here is a simple example. 

\begin{example}\label{example: cohomology n=4 h=2444}
Let $n=4$ and $h=(2,4,4,4)$. Then $w_h = 3124 = s_2 s_1$. From Corollary~\ref{cor: Hess AT KST} it follows that we are looking for faces of the form $F \cap F^\ast$ where $F$ is a reduced Kogan face, $F^*$ is a reduced dual Kogan face, $w(F^\ast)^{-1}w(F)=w_h$, and $\ell(w(F))+\ell(w(F^\ast))=2$. There are four faces $F\cap F^\ast$ satisfying these conditions; see Figure~\ref{fig: n=4 h=2444}. The leftmost face in Figure~\ref{fig: n=4 h=2444} satisfies $w(F)=s_2s_1$ and $w(F^\ast)=id$. The middle two faces satisfy $w(F)=s_1$ and $w(F^\ast)=s_2$. The rightmost face satisfies $w(F)=id$ and $w(F^\ast)=s_1s_2$.
Thus $[\Hess(S,h)]$ is the sum of these four faces.

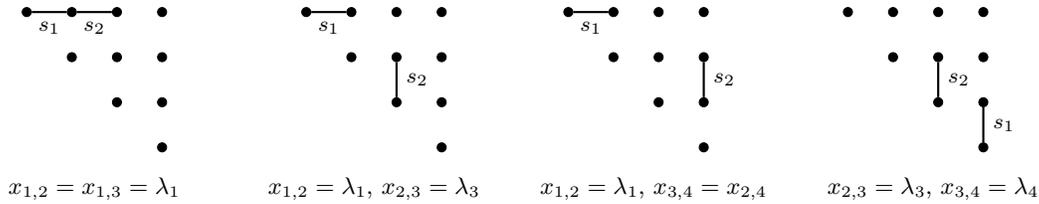
\begin{figure}[h]
\begin{center}
\begin{subfigure}[b]{.24\textwidth}
\begin{center}
    \begin{tikzpicture}[scale=0.6,font=\footnotesize]
    \tikzset{
    solid node/.style={circle,draw,inner sep=1.2,fill=black}
    }
        \node [solid node](01) at (0,3) {};
        \node [solid node](02) at (1,2) {};
        \node [solid node](03) at (2,1) {};
        \node [solid node](04) at (3,0) {};
        \node [solid node](11) at (1,3) {};
        \node [solid node](12) at (2,2) {};
        \node [solid node](13) at (3,1) {};
        \node [solid node](21) at (2,3) {};
        \node [solid node](22) at (3,2) {};
        \node [solid node](31) at (3,3) {};
        \path[thick] (01) edge node[below]{$s_1$} (11);
        \path[thick] (11) edge node[below]{$s_2$} (21);
     
    \end{tikzpicture}
    \end{center}
        \subcaption*{$x_{1,2}=x_{1,3}=\lambda_1$}
    \end{subfigure}
    \begin{subfigure}[b]{.24\textwidth}
        \begin{center}
    \begin{tikzpicture}[scale=0.6,font=\footnotesize]
    \tikzset{
    solid node/.style={circle,draw,inner sep=1.2,fill=black}
    }
        \node [solid node](01) at (0,3) {};
        \node [solid node](02) at (1,2) {};
        \node [solid node](03) at (2,1) {};
        \node [solid node](04) at (3,0) {};
        \node [solid node](11) at (1,3) {};
        \node [solid node](12) at (2,2) {};
        \node [solid node](13) at (3,1) {};
        \node [solid node](21) at (2,3) {};
        \node [solid node](22) at (3,2) {};
        \node [solid node](31) at (3,3) {};
        \path[thick] (01) edge node[below]{$s_1$} (11);
        \path[thick] (12) edge node[right]{$s_2$}(03);
        \end{tikzpicture}
        \end{center}
        \subcaption*{$x_{1,2}=\lambda_1$, $x_{2,3}=\lambda_3$}
    \end{subfigure}
    \begin{subfigure}[b]{.24\textwidth}
    \begin{center}
    \begin{tikzpicture}[scale=0.6,font=\footnotesize]
        \tikzset{
        solid node/.style={circle,draw,inner sep=1.2,fill=black},
        hollow node/.style={circle,draw,inner sep=1.2}
        }
        \node [solid node](01) at (0,3) {};
        \node [solid node](02) at (1,2) {};
        \node [solid node](03) at (2,1) {};
        \node [solid node](04) at (3,0) {};
        \node [solid node](11) at (1,3) {};
        \node [solid node](12) at (2,2) {};
        \node [solid node](13) at (3,1) {};
        \node [solid node](21) at (2,3) {};
        \node [solid node](22) at (3,2) {};
        \node [solid node](31) at (3,3) {};
        \path[thick] (11) edge node[below]{$s_1$} (01);
        \path[thick] (13) edge node[right]{$s_2$} (22);
    \end{tikzpicture}
\end{center}
\subcaption*{$x_{1,2}=\lambda_1$, $x_{3,4}=x_{2,4}$}
\end{subfigure}
    \begin{subfigure}[b]{.24\textwidth}
    \begin{center}
    \begin{tikzpicture}[scale=0.6,font=\footnotesize]
        \tikzset{
        solid node/.style={circle,draw,inner sep=1.2,fill=black},
        hollow node/.style={circle,draw,inner sep=1.2}
        }
        \node [solid node](01) at (0,3) {};
        \node [solid node](02) at (1,2) {};
        \node [solid node](03) at (2,1) {};
        \node [solid node](04) at (3,0) {};
        \node [solid node](11) at (1,3) {};
        \node [solid node](12) at (2,2) {};
        \node [solid node](13) at (3,1) {};
        \node [solid node](21) at (2,3) {};
        \node [solid node](22) at (3,2) {};
        \node [solid node](31) at (3,3) {};
        \path[thick] (12) edge node[right]{$s_2$} (03);
        \path[thick] (13) edge node[right]{$s_1$} (04);
    \end{tikzpicture}
\end{center}
\subcaption*{$x_{2,3}=\lambda_3$, $x_{3,4}=\lambda_4$}
\end{subfigure}
    \end{center}
    \caption{Faces corresponding to $[\Hess(S,h)]$ for $h=(2,4,4,4)$}\label{fig: n=4 h=2444}
\end{figure}

 \end{example}

\section{Combinatorial formulas for the volume of faces of $\GZ(\lambda)$}\label{sec: combinatorics}

The two results of this section are as follows; we will use these results in the next section on volume polynomials. First, we give a combinatorial formula for the volume of a face of the Gelfand-Zetlin polytope which is explicitly expressed as a polynomial in the $\alpha_i = \lambda_i - \lambda_{i+1}$, $i=1,\ldots,n-1$, and whose coefficients count certain combinatorial objects (Proposition~\ref{prop:face of GZ}). In particular, the coefficients are manifestly positive. Our approach is based on the work of Postnikov in \cite{Post} in which he gives a similar formula for the volume of the Gelfand-Zetlin polytope in terms of \textbf{standard shifted Young tableau}.  Thus, our formula can be viewed as a generalization of Postnikov's work to the faces of the Gelfand-Zetlin polytope. Second, we generalize a result of Kiritchenko, Smirnov, and Timorin, which gives a linear relation between the volumes of four faces of $\GZ(\lambda)$ which are obtained by intersecting a larger face with four closely related hyperplanes.

We begin with the terminology required to state Postnikov's result. 
A {\bf standard shifted Young tableau} of triangular shape $(n,n - 1,\ldots,1)$ is a bijective map $T\colon \{(i, j) \mid 1\leq i\leq j \leq n\} \to \{1,\ldots,{n+1 \choose 2}\}$, which is increasing in the rows and the columns, i.e. $T((i, j)) < T((i, j + 1))$ and $T((i, j)) < T((i+1, j))$ whenever the entries are defined.
We say that the {\bf diagonal vector} of such a tableau $T$ is the vector obtained by reading the entries along the main diagonal, i.e. $\mathrm{diag}(T) = (d_1,\ldots,d_n) :=
(T(1,1),T(2,2),...,T(n,n))$. See Figure~\ref{figure: standard shifted Young}.

\begin{figure}[t]
\begin{center}
\begin{subfigure}[b]{.4\textwidth}
\centering
    \begin{tikzpicture}[scale=.53]
    
        \draw (0,0)--(0,-1)--(1,-1)--(1,-2)--(2,-2)--(2,-3)--(3,-3)--(3,-4)--(4,-4)--(4,0)--cycle;
        \foreach \x in {1,2,3}
        \draw (\x,0)--(\x,-\x)--(4,-\x);
        \draw (0.5,-0.5) node{\tiny{1}};
        \draw (1.5,-0.5) node{\tiny{2}};
        \draw (2.5,-0.5) node{\tiny{4}};
        \draw (3.5,-0.5) node{\tiny{7}};
        \draw (1.5,-1.5) node{\tiny{3}};
        \draw (2.5,-1.5) node{\tiny{5}};
        \draw (3.5,-1.5) node{\tiny{8}};
        \draw (2.5,-2.5) node{\tiny{6}};
        \draw (3.5,-2.5) node{\tiny{9}};
        \draw (3.5,-3.5) node{\tiny{10}};
    \end{tikzpicture}
\end{subfigure}
\end{center} 
\caption{Example of a standard shifted Young tableau. The diagonal vector is $(1,3,6,10)$.}
\label{figure: standard shifted Young} 
\end{figure}

In~\cite{Post}, Postnikov showed that $\GZ(\lambda)$ can be divided into products of simplices each of which corresponds to a standard shifted Young tableau of triangular shape $(n,n-1,\ldots,2,1)$. 

\begin{theorem} (\cite[Theorem 15.1]{Post}\label{theorem:Post})
    The volume of the Gelfand--Zetlin polytope $\GZ(\lambda)$ is
    \begin{equation*}
        \vol(\GZ(\lambda))=\sum_{p_1,\ldots,p_{n-1}\geq 0} N(p_1,\ldots,p_{n-1})\frac{\alpha_1^{p_1}}{p_1!}\frac{\alpha_2^{p_2}}{p_2!} \cdots \frac{\alpha_{n-1}^{p_{n-1}}}{p_{n-1}!},
    \end{equation*} where $\alpha_j=\lambda_j-\lambda_{j+1}$ for $j=1,\ldots,n-1$, and the coefficient $N(p_1,\ldots,p_{n-1})$ is equal to the number of standard shifted Young tableaux $T$ of triangular shape $(n, n-1, \ldots, 2,1)$ with diagonal vector equal to $\mathrm{diag}(T)=(1,p_1+2,p_1+p_2+3,\ldots,p_1+\cdots+p_{n-1}+n)$.
\end{theorem}

Note that each coordinate $x_{i,j}$ corresponds to the box on the $i$-th row and the $j$-th column in the shifted Young diagram of the triangular shape $(n,n-1,\ldots,1)$.
It is useful to see an example. 

\begin{example}
Let $n=3$. 
The Gelfand-Zetlin polytope $\GZ(\lambda_1,\lambda_2,\lambda_3)$ can be divided into two regions $$\{\lambda_1\geq x_{1,2}\geq \lambda_2\geq x_{1,3}\geq x_{2,3}\geq \lambda_3\}\text{ and }\{\lambda_1\geq x_{1,2}\geq x_{1,3}\geq \lambda_2\geq x_{2,3}\geq \lambda_3\}$$ which are the products of two simplices $(\lambda_1 - \lambda_2)\Delta^1\times (\lambda_2-\lambda_3)\Delta^2$ and  $(\lambda_1-\lambda_2)\Delta^2\times (\lambda_2-\lambda_3)\Delta^1$, respectively. Hence we have $$\vol(\GZ(\lambda_1,\lambda_2,\lambda_3)) = \frac{(\lambda_1-\lambda_2)^2(\lambda_2 -\lambda_3)}{2}+\frac{(\lambda_1-\lambda_2)(\lambda_2-\lambda_3)^2}{2}.$$ Furthermore, each of the regions corresponds to a standard shifted Young tableau of the triangular shape $(3,2,1)$. See Figure~\ref{fig:ex of SSYT}.  

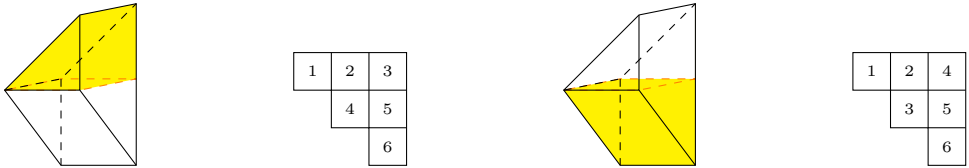
\begin{figure}[h]
\begin{center}
\begin{subfigure}[b]{.24\textwidth}
\centering
\begin{tikzpicture}[scale=.5]
	\fill[yellow] (0,0)--(2,0)--(3.5,0.3)--(3.5,2.3)--(2,2)--cycle;
	\draw[orange,dashed] (0,0)--(2,0)--(3.5,0.3)--(1.5,0.3)--cycle;
	\draw (0,0)--(1.5,-2)--(3.5,-2)--(3.5,2.3)--(2,2)--cycle;
	\draw (0,0)--(2,0)--(2,2);
	\draw (2,0)--(3.5,-2);
	\draw[dashed] (0,0)--(1.5,0.3)--(3.5,2.3);
	\draw[dashed] (1.5,0.3)--(1.5,-2);
\end{tikzpicture}
\end{subfigure}
\begin{subfigure}[b]{.24\textwidth}
\centering
    \begin{tikzpicture}[scale=.5]
        \draw (0,0)--(0,-1)--(1,-1)--(1,-2)--(2,-2)--(2,-3)--(3,-3)--(3,0)--cycle;
        \foreach \x in {1,2}
        \draw (\x,0)--(\x,-\x)--(3,-\x);
        \draw (0.5,-0.5) node{\tiny{1}};
        \draw (1.5,-0.5) node{\tiny{2}};
        \draw (2.5,-0.5) node{\tiny{3}};
        \draw (1.5,-1.5) node{\tiny{4}};
        \draw (2.5,-1.5) node{\tiny{5}};
        \draw (2.5,-2.5) node{\tiny{6}};
    \end{tikzpicture}
\end{subfigure}
\begin{subfigure}[b]{.24\textwidth}
\centering
\begin{tikzpicture}[scale=.5]
	\fill[yellow] (0,0)--(1.5,0.3)--(3.5,0.3)--(3.5,2.3)--(3.5,-2)--(1.5,-2)--cycle;
	\draw[orange,dashed] (0,0)--(2,0)--(3.5,0.3)--(1.5,0.3)--cycle;
	\draw (0,0)--(1.5,-2)--(3.5,-2)--(3.5,2.3)--(2,2)--cycle;
	\draw (0,0)--(2,0)--(2,2);
	\draw (2,0)--(3.5,-2);
	\draw[dashed] (0,0)--(1.5,0.3)--(3.5,2.3);
	\draw[dashed] (1.5,0.3)--(1.5,-2);
\end{tikzpicture}
\end{subfigure}
\begin{subfigure}[b]{.24\textwidth}
\centering
    \begin{tikzpicture}[scale=.5]
        \draw (0,0)--(0,-1)--(1,-1)--(1,-2)--(2,-2)--(2,-3)--(3,-3)--(3,0)--cycle;
        \foreach \x in {1,2}
        \draw (\x,0)--(\x,-\x)--(3,-\x);
        \draw (0.5,-0.5) node{\tiny{1}};
        \draw (1.5,-0.5) node{\tiny{2}};
        \draw (2.5,-0.5) node{\tiny{4}};
        \draw (1.5,-1.5) node{\tiny{3}};
        \draw (2.5,-1.5) node{\tiny{5}};
        \draw (2.5,-2.5) node{\tiny{6}};
    \end{tikzpicture}
\end{subfigure}
\caption{Standard shifted Young tableaux of the triangular shape~$(3,2,1)$ and their corresponding regions}\label{fig:ex of SSYT}
\end{center}
\end{figure}

\end{example}

It turns out that Theorem~\ref{theorem:Post} can be generalized: every face of $\GZ(\lambda)$ can be divided into products of simplices and we can relate them to shifted tableaux, as follows. 
For a face $F$ of $\GZ(\lambda)$, we let 
$$\mathcal{H}(F)=\{(i,j,k,\ell)\mid x_{i,j}=x_{k,\ell}\text{ in }F \text{ and }(k,\ell) \text{ is }(i+1,j) \text{ or }(i,j+1) \}.$$ 
Recall that each box $(i,j)$ of a shifted Young diagram corresponds to a variable, namely $x_{i,j}$. We say a function $T: \{(i,j) \, \mid \, 1 \leq i \leq j \leq n\} \to \{ 1, 2, \ldots, n+\dim F \}$ is a \textbf{shifted Young tableau associated to $F$} if the assignment $T$ is weakly increasing along both rows and columns, and $T((i,j)) = T((i, j+1))$ if and only if $(i,j, i, j+1) \in \mathcal{H}(F)$ (i.e. $x_{i,j} = x_{i, j+1}$ in $F$), and $T((i,j)) = T((i+1, j))$ if and only if $(i,j,i+1, j) \in \mathcal{H}(F)$ (i.e. $x_{i,j} = x_{i+1,j}$ in $F$). 
See Figure~\ref{fig:shifted tableaux} for an example.

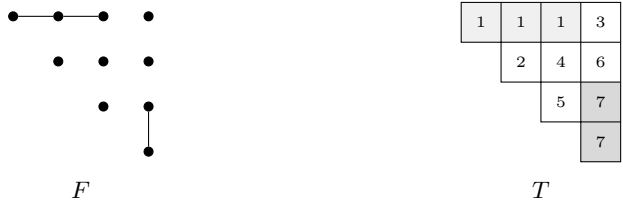
\begin{figure}[h]
\begin{center}
   \begin{subfigure}[b]{.4\textwidth}
\begin{center}
    \begin{tikzpicture}[scale=0.6,font=\footnotesize]
        \tikzset{
        solid node/.style={circle,draw,inner sep=1.2,fill=black},
        hollow node/.style={circle,draw,inner sep=1.2}
        }
        \node [solid node](01) at (0,3) {};
        \node [solid node](02) at (1,2) {};
        \node [solid node](03) at (2,1) {};
        \node [solid node](04) at (3,0) {};
        \node [solid node](11) at (1,3) {};
        \node [solid node](12) at (2,2) {};
        \node [solid node](13) at (3,1) {};
        \node [solid node](21) at (2,3) {};
        \node [solid node](22) at (3,2) {};
        \node [solid node](31) at (3,3) {};
        \path (11) edge (01);
        \path (13) edge (04);
        \path (21) edge (11);
    \end{tikzpicture}
\end{center}
\caption*{$F$}
\end{subfigure}
\begin{subfigure}[b]{.4\textwidth}

\begin{center} 
    \begin{tikzpicture}[scale=.53]
        \fill[gray!20] (0,0)--(0,-1)--(3,-1)--(3,0)--cycle;
        \fill[gray!50] (3,-2)--(3,-4)--(4,-4)--(4,-2)--cycle;
        \draw (0,0)--(0,-1)--(1,-1)--(1,-2)--(2,-2)--(2,-3)--(3,-3)--(3,-4)--(4,-4)--(4,0)--cycle;
        \foreach \x in {1,2,3}
        \draw (\x,0)--(\x,-\x)--(4,-\x);
        \draw (0.5,-0.5) node{\tiny{1}};
        \draw (1.5,-0.5) node{\tiny{1}};
        \draw (2.5,-0.5) node{\tiny{1}};
        \draw (3.5,-0.5) node{\tiny{3}};
        \draw (1.5,-1.5) node{\tiny{2}};
        \draw (2.5,-1.5) node{\tiny{4}};
        \draw (3.5,-1.5) node{\tiny{6}};
        \draw (2.5,-2.5) node{\tiny{5}};
        \draw (3.5,-2.5) node{\tiny{7}};
        \draw (3.5,-3.5) node{\tiny{7}};
    \end{tikzpicture}
    \caption*{$T$} 
    \end{center} 
\end{subfigure}

\end{center}
\caption{Example of a shifted tableaux $T$ associated with a face $F$ of $\GZ(\lambda_1,\ldots,\lambda_4)$}\label{fig:shifted tableaux}
\end{figure}

Using ideas similar to those for \cite[Theorem~15.1]{Post} we can prove the following.

\begin{proposition} \label{prop:face of GZ}
The volume of a face $F$ of the Gelfand--Zetlin polytope $\GZ(\lambda)$ is given by
\[
\vol(F)=\sum_{p_1,\dots,p_{n-1}\ge 0}N_F(p_1,\dots,p_{n-1})\frac{\alpha_1^{p_1}}{p_1!}\cdots\frac{\alpha_{n-1}^{p_{n-1}}}{p_{n-1}!}
\]
where $N_F(p_1,\dots,p_{n-1})$ is the number of shifted tableaux $T$ associated to $F$ with the diagonal vector $\mathrm{diag}(T)=(1,p_1+2,p_1+p_2+3,\ldots,p_1+\cdots+p_{n-1}+n)$.
In particular, $\vol(F)$ is a polynomial with nonnegative coefficients in the variables $\alpha_1,\ldots,\alpha_{n-1}$.
\end{proposition}

\begin{proof}
    Let us subdivide $F$ into parts by the hyperplanes $x_{i,j}=x_{k,\ell},$ for all $(i,j,k,\ell)\not\in \mathcal{H}(F).$ A region of this subdivision of the face $F$ corresponds to a choice of a total ordering of the $x_{i,j}$ compatible with all inequalities. Such orderings are in one-to-one correspondence with {shifted tableaux associated with the face $F$} of $\GZ(\lambda).$ {That is, for a given region, the elements equal to $\lambda_1$ correspond to the boxes containing a~$1$, and they are the maximal elements in the region. The second maximal elements in the region correspond to the boxes containing a~$2$, and so on.} For a tableau $T$ with the diagonal vector $\mathrm{diag}(T) = (d_1,\ldots,d_n)$, the region of $F$ associated with~$T$ is isomorphic to
    \begin{equation}\label{eq:face associated with T}
        Y_T=\left\{(y_1 > \cdots > y_{n_F}) \,\middle|\, y_{d_i}=\lambda_i \text{ for }i=1,\ldots,n\right\},
    \end{equation} where $n_F=n+\dim F$. Note that each $x_{i,j}$ corresponds to $y_{T(i,j)}$, and $x_{i,j}\geq x_{k,\ell}$ if and only if $T(i,j)\leq T(k,\ell)$. Since $Y_T$ is isomorphic to the direct product of simplices $\alpha_1\Delta^{p_1}\times\cdots\times \alpha_{n-1}\Delta^{p_{n-1}},$ where $\alpha_i=\lambda_i-\lambda_{i+1}$ and $p_i=d_{i+1}-d_i-1$, the volume of $Y_T$ equals $\frac{\alpha_1^{p_i}}{p_1!}\cdots\frac{\alpha_{n-1}^{p_{n-1}}}{p_{n-1}!}.$
    Thus the volume $\vol(F)$ can be written as the sum of these expressions over shifted tableaux associated with the face $F$.
\end{proof}

\begin{example}
    Let $n=4$ and $F$ be the face of $\GZ(\lambda)$ defined by $x_{1,1}=x_{1,2}=x_{1,3}$ and $x_{3,4}=x_{4,4}$. Then $n_F=n+\dim F=7$ and there are seven shifted tableaux associated with~$F$ of the triangular shape $(4,3,2,1)$ as in Figure~\ref{fig:shifted rim hook F}.
    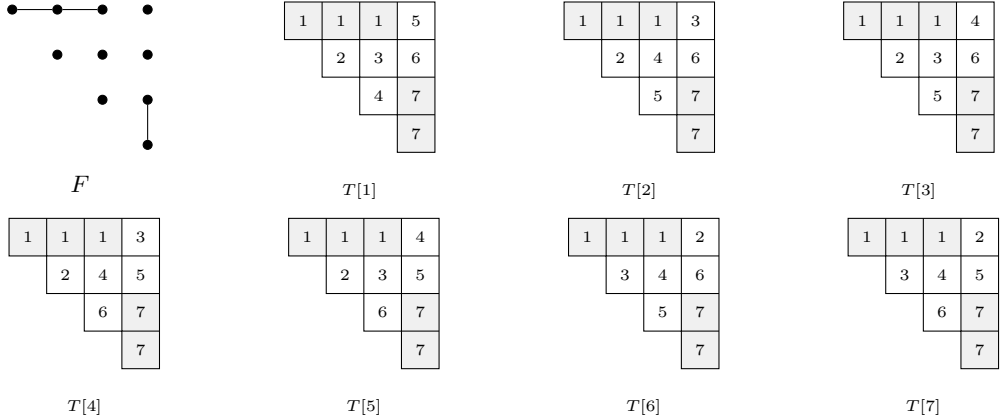
\begin{figure}[h]
       \begin{subfigure}[b]{.24\textwidth}
\begin{center}
    \begin{tikzpicture}[scale=0.6,font=\footnotesize]
        \tikzset{
        solid node/.style={circle,draw,inner sep=1.2,fill=black},
        hollow node/.style={circle,draw,inner sep=1.2}
        }
        \node [solid node](01) at (0,3) {};
        \node [solid node](02) at (1,2) {};
        \node [solid node](03) at (2,1) {};
        \node [solid node](04) at (3,0) {};
        \node [solid node](11) at (1,3) {};
        \node [solid node](12) at (2,2) {};
        \node [solid node](13) at (3,1) {};
        \node [solid node](21) at (2,3) {};
        \node [solid node](22) at (3,2) {};
        \node [solid node](31) at (3,3) {};
        \path (11) edge (01);
        \path (13) edge (04);
        \path (21) edge (11);
    \end{tikzpicture}
\end{center}
\caption*{$F$}
\end{subfigure}
       \begin{subfigure}[b]{.24\textwidth}
\begin{center}
    \begin{tikzpicture}[scale=.5]
        \fill[gray!20] (0,0)--(0,-1)--(3,-1)--(3,0)--cycle;
        \fill[gray!20] (3,-2)--(3,-4)--(4,-4)--(4,-2)--cycle;
        \draw (0,0)--(0,-1)--(1,-1)--(1,-2)--(2,-2)--(2,-3)--(3,-3)--(3,-4)--(4,-4)--(4,0)--cycle;
        \foreach \x in {1,2,3}
        \draw (\x,0)--(\x,-\x)--(4,-\x);
        \draw (0.5,-0.5) node{\tiny{1}};
        \draw (1.5,-0.5) node{\tiny{1}};
        \draw (2.5,-0.5) node{\tiny{1}};
        \draw (3.5,-0.5) node{\tiny{5}};
        \draw (1.5,-1.5) node{\tiny{2}};
        \draw (2.5,-1.5) node{\tiny{3}};
        \draw (3.5,-1.5) node{\tiny{6}};
        \draw (2.5,-2.5) node{\tiny{4}};
        \draw (3.5,-2.5) node{\tiny{7}};
        \draw (3.5,-3.5) node{\tiny{7}};
    \end{tikzpicture}
\end{center}
\caption*{\tiny$T[1]$}
\end{subfigure}
       \begin{subfigure}[b]{.24\textwidth}
\begin{center}
    \begin{tikzpicture}[scale=.5]
        \fill[gray!20] (0,0)--(0,-1)--(3,-1)--(3,0)--cycle;
        \fill[gray!20] (3,-2)--(3,-4)--(4,-4)--(4,-2)--cycle;
        \draw (0,0)--(0,-1)--(1,-1)--(1,-2)--(2,-2)--(2,-3)--(3,-3)--(3,-4)--(4,-4)--(4,0)--cycle;
        \foreach \x in {1,2,3}
        \draw (\x,0)--(\x,-\x)--(4,-\x);
        \draw (0.5,-0.5) node{\tiny{1}};
        \draw (1.5,-0.5) node{\tiny{1}};
        \draw (2.5,-0.5) node{\tiny{1}};
        \draw (3.5,-0.5) node{\tiny{3}};
        \draw (1.5,-1.5) node{\tiny{2}};
        \draw (2.5,-1.5) node{\tiny{4}};
        \draw (3.5,-1.5) node{\tiny{6}};
        \draw (2.5,-2.5) node{\tiny{5}};
        \draw (3.5,-2.5) node{\tiny{7}};
        \draw (3.5,-3.5) node{\tiny{7}};
    \end{tikzpicture}
\end{center}
\caption*{\tiny$T[2]$}
\end{subfigure}
       \begin{subfigure}[b]{.24\textwidth}
\begin{center}
    \begin{tikzpicture}[scale=.5]
        \fill[gray!20] (0,0)--(0,-1)--(3,-1)--(3,0)--cycle;
        \fill[gray!20] (3,-2)--(3,-4)--(4,-4)--(4,-2)--cycle;
        \draw (0,0)--(0,-1)--(1,-1)--(1,-2)--(2,-2)--(2,-3)--(3,-3)--(3,-4)--(4,-4)--(4,0)--cycle;
        \foreach \x in {1,2,3}
        \draw (\x,0)--(\x,-\x)--(4,-\x);
        \draw (0.5,-0.5) node{\tiny{1}};
        \draw (1.5,-0.5) node{\tiny{1}};
        \draw (2.5,-0.5) node{\tiny{1}};
        \draw (3.5,-0.5) node{\tiny{4}};
        \draw (1.5,-1.5) node{\tiny{2}};
        \draw (2.5,-1.5) node{\tiny{3}};
        \draw (3.5,-1.5) node{\tiny{6}};
        \draw (2.5,-2.5) node{\tiny{5}};
        \draw (3.5,-2.5) node{\tiny{7}};
        \draw (3.5,-3.5) node{\tiny{7}};
    \end{tikzpicture}
\end{center}
\caption*{\tiny$T[3]$}
\end{subfigure}
\vspace{.2cm}

       \begin{subfigure}[b]{.24\textwidth}
\begin{center}
    \begin{tikzpicture}[scale=.5]
        \fill[gray!20] (0,0)--(0,-1)--(3,-1)--(3,0)--cycle;
        \fill[gray!20] (3,-2)--(3,-4)--(4,-4)--(4,-2)--cycle;
        \draw (0,0)--(0,-1)--(1,-1)--(1,-2)--(2,-2)--(2,-3)--(3,-3)--(3,-4)--(4,-4)--(4,0)--cycle;
        \foreach \x in {1,2,3}
        \draw (\x,0)--(\x,-\x)--(4,-\x);
        \draw (0.5,-0.5) node{\tiny{1}};
        \draw (1.5,-0.5) node{\tiny{1}};
        \draw (2.5,-0.5) node{\tiny{1}};
        \draw (3.5,-0.5) node{\tiny{3}};
        \draw (1.5,-1.5) node{\tiny{2}};
        \draw (2.5,-1.5) node{\tiny{4}};
        \draw (3.5,-1.5) node{\tiny{5}};
        \draw (2.5,-2.5) node{\tiny{6}};
        \draw (3.5,-2.5) node{\tiny{7}};
        \draw (3.5,-3.5) node{\tiny{7}};
    \end{tikzpicture}
\end{center}
\caption*{\tiny$T[4]$}
\end{subfigure}
       \begin{subfigure}[b]{.24\textwidth}
\begin{center}
    \begin{tikzpicture}[scale=.5]
        \fill[gray!20] (0,0)--(0,-1)--(3,-1)--(3,0)--cycle;
        \fill[gray!20] (3,-2)--(3,-4)--(4,-4)--(4,-2)--cycle;
        \draw (0,0)--(0,-1)--(1,-1)--(1,-2)--(2,-2)--(2,-3)--(3,-3)--(3,-4)--(4,-4)--(4,0)--cycle;
        \foreach \x in {1,2,3}
        \draw (\x,0)--(\x,-\x)--(4,-\x);
        \draw (0.5,-0.5) node{\tiny{1}};
        \draw (1.5,-0.5) node{\tiny{1}};
        \draw (2.5,-0.5) node{\tiny{1}};
        \draw (3.5,-0.5) node{\tiny{4}};
        \draw (1.5,-1.5) node{\tiny{2}};
        \draw (2.5,-1.5) node{\tiny{3}};
        \draw (3.5,-1.5) node{\tiny{5}};
        \draw (2.5,-2.5) node{\tiny{6}};
        \draw (3.5,-2.5) node{\tiny{7}};
        \draw (3.5,-3.5) node{\tiny{7}};
    \end{tikzpicture}
\end{center}
\caption*{\tiny$T[5]$}
\end{subfigure}
       \begin{subfigure}[b]{.24\textwidth}
\begin{center}
    \begin{tikzpicture}[scale=.5]
        \fill[gray!20] (0,0)--(0,-1)--(3,-1)--(3,0)--cycle;
        \fill[gray!20] (3,-2)--(3,-4)--(4,-4)--(4,-2)--cycle;
        \draw (0,0)--(0,-1)--(1,-1)--(1,-2)--(2,-2)--(2,-3)--(3,-3)--(3,-4)--(4,-4)--(4,0)--cycle;
        \foreach \x in {1,2,3}
        \draw (\x,0)--(\x,-\x)--(4,-\x);
        \draw (0.5,-0.5) node{\tiny{1}};
        \draw (1.5,-0.5) node{\tiny{1}};
        \draw (2.5,-0.5) node{\tiny{1}};
        \draw (3.5,-0.5) node{\tiny{2}};
        \draw (1.5,-1.5) node{\tiny{3}};
        \draw (2.5,-1.5) node{\tiny{4}};
        \draw (3.5,-1.5) node{\tiny{6}};
        \draw (2.5,-2.5) node{\tiny{5}};
        \draw (3.5,-2.5) node{\tiny{7}};
        \draw (3.5,-3.5) node{\tiny{7}};
    \end{tikzpicture}
\end{center}
\caption*{\tiny$T[6]$}
\end{subfigure}
       \begin{subfigure}[b]{.24\textwidth}
\begin{center}
    \begin{tikzpicture}[scale=.5]
        \fill[gray!20] (0,0)--(0,-1)--(3,-1)--(3,0)--cycle;
        \fill[gray!20] (3,-2)--(3,-4)--(4,-4)--(4,-2)--cycle;
        \draw (0,0)--(0,-1)--(1,-1)--(1,-2)--(2,-2)--(2,-3)--(3,-3)--(3,-4)--(4,-4)--(4,0)--cycle;
        \foreach \x in {1,2,3}
        \draw (\x,0)--(\x,-\x)--(4,-\x);
        \draw (0.5,-0.5) node{\tiny{1}};
        \draw (1.5,-0.5) node{\tiny{1}};
        \draw (2.5,-0.5) node{\tiny{1}};
        \draw (3.5,-0.5) node{\tiny{2}};
        \draw (1.5,-1.5) node{\tiny{3}};
        \draw (2.5,-1.5) node{\tiny{4}};
        \draw (3.5,-1.5) node{\tiny{5}};
        \draw (2.5,-2.5) node{\tiny{6}};
        \draw (3.5,-2.5) node{\tiny{7}};
        \draw (3.5,-3.5) node{\tiny{7}};
    \end{tikzpicture}
\end{center}
\caption*{\tiny$T[7]$}
\end{subfigure}
\caption{Shifted tableaux associated with $F$ of triangular shape $(4,3,2,1)$}\label{fig:shifted rim hook F}
    \end{figure}
    By Proposition~\ref{prop:face of GZ}, we can compute the volume of the region associated with $T[i]$ for $i=1,\ldots,7$ as follows.
    \begin{equation*}
    \begin{split}
        &\vol(Y_{T[1]})=\frac{\alpha_2^1\alpha_3^2}{2!},\,\,\,\vol(Y_{T[2]})=\vol(Y_{T[3]})=\frac{\alpha_2^2\alpha_3}{2!}\\
        & \vol(Y_{T[4]})=\vol(Y_{T[5]})=\frac{\alpha_2^3}{3!},\,\,\, \vol(Y_{T[6]})=\alpha_1\alpha_2\alpha_3, \,\,\,\vol(Y_{T[7]})=\frac{\alpha_1\alpha_2^2}{2!}
    \end{split}
    \end{equation*}
    Also note that $T[2]$ and $T[3]$ have the same diagonal entries, as do $T[4]$ and $T[5]$; hence they contribute to the same monomial term in the formula. Putting this together, we obtain that the volume of the face $F$ is equal to $$\vol(F)=\frac{\alpha_2^1\alpha_3^2}{2!}+2\frac{\alpha_2^2\alpha_3}{2!}+2\frac{\alpha_2^3}{3!}+\alpha_1\alpha_2\alpha_3+\frac{\alpha_1\alpha_2^2}{2!}.$$
\end{example}

We now state and prove our second combinatorial result. We need some preliminaries. Let $F$ be a face of $\GZ(\lambda)$, specified by equations tabulated in the set $\mathcal{H}(F)$ as above. We say that a variable $x_{i,j}$ is \textbf{isolated} in $F$ if $x_{i,j}$ does not appear in any of the equations as given in $\mathcal{H}(F)$ which specify $F$ (or, equivalently, the dot at location $(i,j)$ is not contained in any edge of the face diagram of $F$). Recall that our variables are $\{x_{i,j}\}$ where $1 \leq i < j \leq n$ and we denote $\lambda_i:=x_{i,i}$ for $1\leq i\leq n$. In what follows, if a pair $(i', j')$ of integers do not satisfy the condition $1 \leq i' \leq j' \leq n$ then we say that the indices are ``out of range''.

\begin{proposition}\label{prop:x-relation}
	Let $F$ be a face of $\GZ(\lambda)$. Let $x_{i,j}$ for $1 \leq i < j \leq n$ be isolated in $F$. 
 Assume also that each of the four variables $x_{i, j-1}$, $x_{i-1, j}$, $x_{i+1, j}$ and $x_{i, j+1}$ is either isolated or its indices are out of range. Then 

	\begin{equation*}\label{eq:x-relation for face 1}
    \begin{split}
    & \vol(F \cap \{x_{i,j} = x_{i,j-1} \}) + \vol(F \cap \{x_{i,j} = x_{i-1,j} \})\\
&\qquad=\vol(F \cap \{x_{i,j} = x_{i+1,j} \}) + \vol(F \cap \{x_{i,j} = x_{i,j+1} \}),
\end{split}
    \end{equation*}
    where a term should be set to $0$ if it involves indices that are out of range. 
 \end{proposition}

The above is a generalization of \cite[Proposition 3.2]{KST12}, which is the special case of Proposition~\ref{prop:x-relation} when $F=\GZ(\lambda)$ is the whole Gelfand-Zetlin polytope. 

Before proving the proposition we give the idea of what is happening. Let $F$ be a face of $\GZ(\lambda)$. Suppose, as stated in the hypothesis above, that we have a collection of isolated variables $x_{i,j}$ and its four ``neighbors'' $x_{i, j-1}$, $x_{i, j+1}$, $x_{i-1, j}$ and $x_{i+1,j}$. These five variables can be visualized as appearing in an arrangement (within the full arrangement of Gelfand-Zetlin variables indicated in~\eqref{eq:GZ}) as indicated in the figure below, where we have drawn both the variables as well as the piece of the face diagram corresponding to them. Note that there are no edges in the face diagram, since $x_{i,j}$ is assumed to be isolated; moreover, these five dots also do not touch any other edge of the face diagram. See Figure~\ref{fig: isolated}.

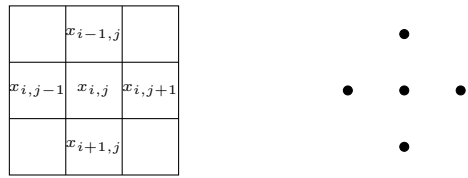
\begin{figure}[h]
	\begin{tikzpicture}[scale=.75]
	\draw (0,0)--(3,0)--(3,3)--(0,3)--cycle;
	\foreach \x in {1,2}
	{\draw (\x,0)--(\x,3);
	\draw(0,\x)--(3,\x);
	}
	\draw (1.5,1.5) node{\tiny${x_{i,j}}$};
	\draw (0.5,1.5) node{\tiny${x_{i, j-1}}$};
	\draw (1.5,2.5) node{\tiny${x_{i-1,j}}$};
	\draw (1.5,0.5) node{\tiny${x_{i+1,j}}$};
	\draw (2.5,1.5) node{\tiny${x_{i,j+1}}$};
	
	\tikzset{
        solid node/.style={circle,draw,inner sep=1.2,fill=black},
        hollow node/.style={circle,draw,inner sep=1.2}
        }
        \node [solid node](02) at (1+5,2-.5) {};
        \node [solid node](03) at (2+5,1-.5) {};
        \node [solid node](12) at (2+5,2-.5) {};
        \node [solid node](21) at (2+5,3-.5) {};
        \node [solid node](22) at (3+5,2-.5) {};

	\end{tikzpicture}
	\caption{An isolated variable $x_{i,j}$ and its (isolated) neighbors.}\label{fig: isolated} 
	\end{figure}

We may now consider four different faces of $\GZ(\lambda)$, obtained by intersecting $F$ with an additional hyperplane corresponding to the equations $x_{i,j} = x_{i,j-1}, x_{i,j} = x_{i-1,j}, x_{i,j} = x_{i, j+1},$ and $x_{i,j} = x_{i+1, j}$. We can think of these pictorially as well; for instance, the equation $x_{i,j}=x_{i,j-1}$ can be  represented by the picture in Figure~\ref{fig: piece}. 

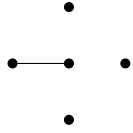
\begin{figure}[h]
	\begin{tikzpicture}[scale=.75]

	\tikzset{
        solid node/.style={circle,draw,inner sep=1.2,fill=black},
        hollow node/.style={circle,draw,inner sep=1.2}
        }
        \node [solid node](02) at (1+5,2-1) {};
        \node [solid node](03) at (2+5,1-1) {};
        \node [solid node](12) at (2+5,2-1) {};
        \node [solid node](21) at (2+5,3-1) {};
        \node [solid node](22) at (3+5,2-1) {}; 
        \path (02) edge (12);

	\end{tikzpicture}
	\caption{The equation $x_{i,j}=x_{i,j-1}$ represented as a piece of a face diagram.}\label{fig: piece} 
	\end{figure}

With these ideas in place, we can give the idea of the proposition. We consider the sum of the volumes of the faces obtained by intersecting $F$ with the two hyperplanes $\{x_{i,j} = x_{i,j-1}\}$ and $\{x_{i,j}=x_{i-1,j}\}$, corresponding to the two edges which lie `to the northwest' of the dot at $(i,j)$. We can do the same for the two hyperplanes corresponding to $\{x_{i,j} = x_{i, j+1}\}$ and $\{x_{i,j} = x_{i+1,j}\}$, which correspond to the two edges lying `to the southeast' of $(i,j)$. The point of our proposition is that these two quantities are the same. The argument also works, with minor adjustments, when some of the variables are `out of range', provided that
$x_{i,j}$ is not on the main diagonal in the triangle~\eqref{eq:GZ}.

\begin{proof}[Proof of Proposition~\ref{prop:x-relation}]
We first deal with the case when all variables are defined, i.e., none of the indices are out of range. From Proposition~\ref{prop:face of GZ} it follows that, in order to prove the claim, it would suffice to show that for any fixed $p_1, p_2, \ldots, p_{n-1} \geq 0$, we have that the two quantities
\begin{equation}\label{eq: N LHS} 
N_{F \cap \{x_{i,j} = x_{i,j-1}\}}(p_1, p_2, \ldots, p_{n-1}) + N_{F \cap \{x_{i,j} = x_{i-1,j}\}}(p_1, p_2, \ldots, p_{n-1}) \end{equation}
and
\begin{equation}\label{eq: N RHS} 
N_{F \cap \{x_{i,j} = x_{i+1,j}\}}(p_1, p_2, \ldots, p_{n-1}) + N_{F \cap \{x_{i,j} = x_{i,j+1}\}}(p_1, p_2, \ldots, p_{n-1}) 
\end{equation}  
are equal. 
For this purpose we define the set $\mathrm{ShYT}(F, p_1, p_2, \ldots, p_{n-1})$ to be the set of all shifted Young tableau $T$ associated to $F$ with $\mathrm{diag}(T) = (1, p_1+2, p_1+p_2+3,\ldots, p_1+\cdots+p_{n-1}+n)$. In fact, for the duration of this argument the parameters $p_1, p_2,\ldots,p_{n-1}$ are fixed, so for simplicity we suppress the $p_i$ from the notation and we denote $\mathrm{ShYT}(F) := \mathrm{ShYT}(F, p_1,p_2,\ldots,p_{n-1})$. So it follows immediately that the integer $N_{F'}(p_1, \ldots, p_{n-1})$ for a face $F'$ is by definition 
\[
N_{F'}(p_1, \ldots, p_{n-1}) = \lvert \mathrm{ShYT}(F') \rvert.
\]
Next, note that from the definition of shifted Young tableaux associated to faces, and because each of the four faces appearing in~\eqref{eq: N LHS} and~\eqref{eq: N RHS} have distinct defining equations, the four sets $\mathrm{ShYT}(F \cap \{x_{i,j} = x_{i,j-1}\})$ etc. are all (pairwise) disjoint. From this it follows that to prove that~\eqref{eq: N LHS} is equal to~\eqref{eq: N RHS}, it suffices to prove that there is a bijection 
between
\begin{equation}\label{eq: S LHS} 
\mathrm{ShYT}(F \cap \{x_{i,j} = x_{i, j-1}\}) \bigsqcup 
\mathrm{ShYT}(F \cap \{x_{i,j} = x_{i-1, j}\}) 
\end{equation}
and
\begin{equation}\label{eq: S RHS} 
\mathrm{ShYT}(F \cap \{x_{i,j} = x_{i, j+1}\}) \bigsqcup 
\mathrm{ShYT}(F \cap \{x_{i,j} = x_{i+1, j}\}).
\end{equation}

Let $T$ be a shifted Young tableau in~$\mathrm{ShYT}(F \cap \{x_{i,j} = x_{i, j-1}\})$, $\mathrm{ShYT}(F \cap \{x_{i,j} = x_{i-1, j}\})$, $\mathrm{ShYT}(F \cap \{x_{i,j} = x_{i, j+1}\})$, or $\mathrm{ShYT}(F \cap \{x_{i,j} = x_{i+1, j}\})$. We focus on the entries of $T$ in the five locations $(i,j)$, $(i-1,j)$, $(i,j-1)$, $(i+1,j)$ and $(i,j+1)$, which we represent as in the figure below. 

	\begin{figure}[h]
	\begin{tikzpicture}[scale=.5]
	\draw (0,0)--(3,0)--(3,3)--(0,3)--cycle;
	\foreach \x in {1,2}
	{\draw (\x,0)--(\x,3);
	\draw(0,\x)--(3,\x);
	}
	\draw (1.5,1.5) node{$x$};
	\draw (0.5,1.5) node{$a$};
	\draw (1.5,2.5) node{$b$};
	\draw (1.5,0.5) node{$c$};
	\draw (2.5,1.5) node{$d$};
	\end{tikzpicture}
	\end{figure}
From the definition of shifted tableaux, and from the assumption that the variables $x_{i, j-1}$, $x_{i-1,j}$, $x_{i+1,j}$ and $x_{i,j+1}$ are isolated, it follows that the integers $a$, $b$, $c$, $d$ are pairwise distinct and $x$ is equal to one of $a$, $b$, $c$, and $d$.
In particular we know $a<d$, $b<c$, $b<d$, and $a<c$. On the other hand, the relationships between $a$ and $b$, as well as $c$ and $d$, are not determined by the conditions on shifted tableau. Therefore there are four possible cases to consider: 
\begin{enumerate} 
\item[Case 1:]  $a>b$ and $c<d$. 
\item[Case 2:]  $a>b$ and $c>d$.
\item[Case 3:] $a<b$ and $c<d$.
\item[Case 4:]  $a<b$ and $c>d$.
\end{enumerate} 	
It is not hard to check the following:
\begin{itemize}
\item $T$ belongs to $\mathrm{ShYT}(F \cap \{x_{i,j} = x_{i, j-1}\})$ if and only if $x=a$ and $T$ satisfies cases 1 or 2;
\item $T$ belongs to $\mathrm{ShYT}(F \cap \{x_{i,j} = x_{i-1, j}\})$ if and only if $x=b$ and $T$ satisfies cases 3 or 4;
\item $T$ belongs to $\mathrm{ShYT}(F \cap \{x_{i,j} = x_{i+1, j}\})$ if and only if $x=c$ and $T$ satisfies cases 1 or 3; and
\item $T$ belongs to $\mathrm{ShYT}(F \cap \{x_{i,j} = x_{i, j+1}\})$ if and only if $x=d$ and $T$ satisfies cases 2 or 4.
\end{itemize}  See Figures~\ref{fig: PsiL} and~\ref{fig: PsiR}.

\begin{figure}[h]
	\centering
	\begin{subfigure}{.4\textwidth}
	\centering
	\begin{tikzpicture}[scale=.5]
	\draw (0,0)--(3,0)--(3,3)--(0,3)--cycle;
	\foreach \x in {1,2}
	{\draw (\x,0)--(\x,3);
	\draw(0,\x)--(3,\x);
	}
	\draw (1.5,1.5) node{$a$};
	\draw (0.5,1.5) node{$a$};
	\draw (1.5,2.5) node{$b$};
	\draw (1.5,0.5) node{$c$};
	\draw (2.5,1.5) node{$d$};
	\end{tikzpicture}
	\caption{$x=a$}
	\end{subfigure}
	\begin{subfigure}{.4\textwidth}
	\centering
	\begin{tikzpicture}[scale=.5]
	\draw (0,0)--(3,0)--(3,3)--(0,3)--cycle;
	\foreach \x in {1,2}
	{\draw (\x,0)--(\x,3);
	\draw(0,\x)--(3,\x);
	}
	\draw (1.5,1.5) node{$b$};
	\draw (0.5,1.5) node{$a$};
	\draw (1.5,2.5) node{$b$};
	\draw (1.5,0.5) node{$c$};
	\draw (2.5,1.5) node{$d$};
	\end{tikzpicture}
	\caption{$x=b$}
	\end{subfigure}
\caption{Pieces of shifted tableau associated to $F \cap \{x_{i,j}=x_{i,j-1}\}$ and $F \cap \{x_{i,j} = x_{i-1,j}\}$.}\label{fig: PsiL} 
\end{figure} 

\begin{figure}[h]
	\centering

	\begin{subfigure}{.4\textwidth}
	\centering
	\begin{tikzpicture}[scale=.5]
	\draw (0,0)--(3,0)--(3,3)--(0,3)--cycle;
	\foreach \x in {1,2}
	{\draw (\x,0)--(\x,3);
	\draw(0,\x)--(3,\x);
	}
	\draw (1.5,1.5) node{$c$};
	\draw (0.5,1.5) node{$a$};
	\draw (1.5,2.5) node{$b$};
	\draw (1.5,0.5) node{$c$};
	\draw (2.5,1.5) node{$d$};
	\end{tikzpicture}
	\caption{$x=c$}
	\end{subfigure}
	\begin{subfigure}{.4\textwidth}
	\centering
	\begin{tikzpicture}[scale=.5]
	\draw (0,0)--(3,0)--(3,3)--(0,3)--cycle;
	\foreach \x in {1,2}
	{\draw (\x,0)--(\x,3);
	\draw(0,\x)--(3,\x);
	}
	\draw (1.5,1.5) node{$d$};
	\draw (0.5,1.5) node{$a$};
	\draw (1.5,2.5) node{$b$};
	\draw (1.5,0.5) node{$c$};
	\draw (2.5,1.5) node{$d$};
	\end{tikzpicture}
	\caption{$x=d$}
	\end{subfigure}
	\caption{Pieces of shifted tableau associated to $F \cap \{x_{i,j}=x_{i+1,j}\}$ and $F \cap \{x_{i,j} = x_{i,j+1}\}$}\label{fig: PsiR}
	\end{figure}

Hence for each $T\in \mathrm{ShYT}(F \cap \{x_{i,j} = x_{i, j-1}\}) \bigsqcup 
\mathrm{ShYT}(F \cap \{x_{i,j} = x_{i-1, j}\})$, if $c<d$, then $T$ can be transformed into a shifted tableau in $\mathrm{ShYT}(F \cap \{x_{i,j} = x_{i+1, j}\})$ by changing $x=c$, and if $c>d$, then $T$ can be transformed into a shifted Young tableau in $\mathrm{ShYT}(F \cap \{x_{i,j} = x_{i, j+1}\})$ by changing $x=d$. Therefore, 
\begin{align*}
&|\mathrm{ShYT}(F \cap \{x_{i,j} = x_{i, j-1}\})|+|\mathrm{ShYT}(F \cap \{x_{i,j} = x_{i-1, j}\})|\\
& \leq |\mathrm{ShYT}(F \cap \{x_{i,j} = x_{i, j+1}\})| + \mathrm{ShYT}(F \cap \{x_{i,j} = x_{i+1, j}\})|.
\end{align*} Similarly, we can show that each $T\in \mathrm{ShYT}(F \cap \{x_{i,j} = x_{i, j+1}\}) \bigsqcup 
\mathrm{ShYT}(F \cap \{x_{i,j} = x_{i+1, j}\})$ can be transformed into a shifted Young tableau in $\mathrm{ShYT}(F \cap \{x_{i,j} = x_{i, j-1}\}) \bigsqcup 
\mathrm{ShYT}(F \cap \{x_{i,j} = x_{i-1, j}\})$ and hence we get 
\begin{align*}
&|\mathrm{ShYT}(F \cap \{x_{i,j} = x_{i, j-1}\})|+|\mathrm{ShYT}(F \cap \{x_{i,j} = x_{i-1, j}\})|\\
& \geq |\mathrm{ShYT}(F \cap \{x_{i,j} = x_{i, j+1}\})| + \mathrm{ShYT}(F \cap \{x_{i,j} = x_{i+1, j}\})|.
\end{align*}

Finally, it is not difficult to make minor adjustments to the considerations above to cover the cases in which either
 $b$ or $d$ is out of range. We leave details to the reader. \end{proof}

\begin{remark}
In fact, Proposition~\ref{prop:x-relation} can be further generalized. Suppose $F$ is a face of $\GZ(\lambda)$ and suppose that $x_{i,j}$ is isolated for $F$. Let $\mathcal{C}_a$, $\mathcal{C}_b$, $\mathcal{C}_c$ and $\mathcal{C}_d$ be the connected components containing the four vertices $(i,j-1)$, $(i-1,j)$, $(i+1,j)$ and $(i,j+1)$ in the face diagram of $F$. Note that since $x_{i-1,j-1}\geq x_{i-1,j}\geq x_{i,j}$ and $x_{i-1,j-1}\geq x_{i,j-1}\geq x_{i,j}$, if $x_{i',j'}=x_{i,j}$ for $i'<i$ and $j'<j$, then $x_{k,\ell}=x_{i,j}$ for $i'\leq k\leq i$ and $j'\leq \ell\leq j$. It is not difficult to adjust the proof of Proposition~\ref{prop:x-relation} to cover the case when each of 
the four vertices $(i-1,j-1)$, $(i-1,j+1)$, $(i+1,j-1)$ and $(i+1,j+1)$ is not contained in any of the sets $\mathcal{C}_a$, $\mathcal{C}_b$, $\mathcal{C}_c$ and $\mathcal{C}_d$. 
\end{remark}

\section{Volume polynomials of regular semisimple Hessenberg varieties}\label{sec: volume}

In this section we analyze the volume polynomials of regular semisimple Hessenberg varieties. We first introduce the volume polynomial of a subvariety of $\Fl(\C^n)$ and then consider the special case when the subvariety is the regular semisimple Hessenberg variety. In this case, some of our previous results \cite{AHMMS, ADGH} links the volume polynomial of $\Hess(S,h)$ to the volume polynomial of the Gelfand-Zetlin polytope $\GZ(\lambda)$. Some explicit computations in small-$n$ cases led us to believe that, firstly, the volume polynomial for $\Hess(S,h)$ should be an appropriate linear combination of the volumes of faces of $\GZ(\lambda)$. Secondly, we suspected that the volume polynomial, when expressed in terms of monomials in the $\alpha_i := \lambda_i - \lambda_{i+1}$, should have non-negative coefficients. The discussion in this section shows that both of the above are true. We take a somewhat expository approach in this section: although not strictly logically necessary, we use suggestive small-$n$ examples to illustrate our motivation for studying these phenomena. After the expository detour we answer in the affirmative the two questions posed above.

We begin with the definition of the volume polynomial associated to a subvariety of $\Fl(\C^n)$. We need some preliminaries. Let $E_i$ denote the $i$-th tautological vector bundle over $\Fl(\C^n)$; namely, $E_i$ is the sub-bundle of the trivial vector bundle $\Fl(\C^n)\times \C^n$ over $\Fl(\C^n)$ whose fiber over a point $V_{\bullet}$ is exactly $V_i$. Let $L_i := E_i/E_{i-1}$ be the quotient line bundle and let $L_i^*$ be the dual line bundle. 
We let $x_i$ denote the first Chern class of $L_i^*$, or equivalently, the negative of the first Chern class of $L_i$ over $\Fl(\C^n)$:
$$
x_i:=-c_1(E_i/E_{i-1}) = - c_1(L_i) = c_1(L_i^*) \ \ \ \textup{ for } i=1,2,\ldots,n.
$$
To each $\lambda=(\lambda_1,\lambda_2,\cdots,\lambda_n) \in \R^n$ and a cohomology class $\alpha\in H^{2(m-d)}(\Fl(\C^n))$, $m=\dim_\C\Fl(\C^n)$, we can assign the real number $\frac{1}{d!}\int_{\Fl(\C^n)} (\lambda_1x_1+\cdots+\lambda_nx_n)^{d}\alpha$, where $(\lambda_1x_1+\cdots+\lambda_nx_n)^{d}$ is an element of the cohomology ring $H^{2d}(\Fl(\C^n))$ and $\int_{\Fl(\C^n)}$ denotes the operator which takes the cap product with the fundamental class of $\Fl(\C^n)$.
Then the {\bf volume polynomial of a cohomology class $\alpha\in H^{2(m-d)}(\Fl(\C^n))$}, viewed as a function of the variables $\lambda_1, \lambda_2, \cdots, \lambda_n$, is defined by the formula
\begin{equation}\label{eq: def vol_lambda}
\vol_\lambda(\alpha):=\frac{1}{d!}\int_{\Fl(\C^n)} (\lambda_1x_1+\cdots+\lambda_nx_n)^{d}\alpha.
\end{equation} 
Note that $\vol_\lambda(\alpha)$ is a homogeneous polynomial of degree $d$ in the variables $\lambda_1,\lambda_2,\cdots,\lambda_n$. Then for an irreducible subvariety $Y$ of $\Fl(\C^n)$, we define $\vol_\lambda(Y) = \vol_\lambda([Y])$, where $[Y]$ is the Poincar\'{e} dual of the cycle of $Y$ in $H^\ast(\Fl(\C^n))$. We refer to $\vol_\lambda(Y)$ as the \textbf{volume polynomial of (the subvariety) $Y \subset \Fl(\C^n)$}. See e.g. \cite{AM} for more discussion related to volume polynomials.

We now recall the relation between the volume polynomials $\vol_\lambda(\Fl(\C^n))$ and $\vol_\lambda(\Hess(S,h))$ of the flag variety and the Hessenberg variety, and the volume of the Gelfand-Zetlin polytope $\GZ(\lambda)$. Indeed, when all of the $\lambda_i$ are integers, there exists a line bundle $L_\lambda$ over the flag variety $\Fl(\C^n)$ such that $c_1(L_\lambda)=\sum_{i=1}^n \lambda_ix_i$.
In fact, we can construct such an $L_\lambda$ explicitly by the formula $L_\lambda:=(L_1^*)^{\lambda_1} \otimes \cdots \otimes (L_n^*)^{\lambda_n}$.
Moreover, if $\lambda_1>\lambda_2>\cdots>\lambda_n$, then $L_\lambda$ is known to be very ample, and it follows readily from the formula~\eqref{eq: def vol_lambda} that, in this case, the quantity $\vol_\lambda(\Y)$ is the degree of $Y$ (in the standard sense of algebraic geometry) with respect to its embedding given by $L_\lambda$ multiplied by $\frac{1}{d!}$ where $d:=\dim_\C Y$. 
This degree, multiplied by $\frac{1}{d!}$, is in turn equal to the (normalized Euclidean) volume\footnote{Here we fix a (translation-invariant) volume form on $\R^d$. If an integer lattice $\Z^d \subset\R^d$ is fixed, we will always choose this volume form to take value~$1$ on the fundamental parallelepiped of $\Z^d$. See \cite{KST12}.}  of a Newton-Okounkov body associated with the line bundle $L_\lambda|_{\Y}$ restricted to $\Y$, by the standard theory of Newton-Okounkov bodies \cite[Section 1.3, cf. also Theorem 3.1 and Corollary 3.2]{KavKho}.  Finally, in the case of the full flag variety $Y=\Fl(\C^n)$, Kaveh has shown that the Gelfand-Zetlin polytope arises as a Newton--Okounkov body of $\Fl(\C^n)$ (for an appropriate choice of valuation) \cite{Kaveh}. 
Hence, in this case, the volume polynomial $\vol_\lambda(\Fl(\C^n))$ is linked through the theory of Newton-Okounkov bodies to the volume of Gelfand-Zetlin polytopes. 
From this discussion it follows that 
\begin{equation}\label{eq: vol Flag is vol GZ}
\vol_\lambda(\Fl(\C^n)) = \vol(\GZ(\lambda)).
\end{equation} 
On the other hand, from \cite[Section 15]{Post} we know that the volume of the Gelfand-Zetlin polytope is given by the formula 
\begin{equation}\label{eq: vol GZ} 
\mathrm{Vol}(\GZ(\lambda)) = \prod_{1 \leq i < j \leq n} \frac{\lambda_i - \lambda_j}{j-i} = 
\frac{1}{1! 2! \cdots (n-1)!} \prod_{1 \leq i < j \leq n} (\lambda_i - \lambda_j).
\end{equation}
Putting~\eqref{eq: vol Flag is vol GZ} and~\eqref{eq: vol GZ} together yields the following formula for the volume polynomial for the flag variety: 
\begin{equation}\label{eq: vol flag}
\vol_\lambda(\Fl(\C^n)) = \frac{1}{1! 2! \cdots (n-1)!} \prod_{1 \leq i < j \leq n} (\lambda_i - \lambda_j).
\end{equation}

As mentioned above, we are also interested in expressing the volume polynomial $\vol_\lambda(\Hess(S,h))$ in terms of the volume $\vol_\lambda(\Fl(\C^n))$ of the flag variety. For this purpose we introduce the following notation. For each $i$, $1 \leq i \leq n$, we view $\lambda_i$ as a variable and set 
\[
\partial_i := \frac{\partial}{\partial \lambda_i}
\]
to be the partial derivative operator with respect to $\lambda_i$. We can now state the following, which  is a consequence of \cite[Theorem 11.3]{AHMMS} (cf. also \cite[Corollary 6.1, Theorem 6.2]{ADGH}). 

\begin{theorem}\label{theorem: AHMMS} (\cite[Theorem 11.3]{AHMMS}, \cite[Corollary 6.1, Theorem 6.2]{ADGH}) 
The volume polynomials $\vol_\lambda(\Hess(S,h))$ and $\vol_\lambda(\Fl(\C^n))$ are related as follows: 
\begin{equation}\label{eq: vol Hess from vol Fl}
\vol_\lambda(\Hess(S,h)) = \left( \prod_{j=1}^{n-1} \prod_{i=h(j)+1}^n (\partial_j - \partial_i) \right) \vol_\lambda(\Fl(\C^n))
\end{equation}
where we take the convention that if $h(j)=n$ then $\prod_{i=h(j)+1}^n (\partial_j - \partial_i) = 1$. 
\end{theorem}

\begin{remark} 
The discussion in \cite{AHMMS} is given in the language of Poincar\'e duality algebras, and \cite[Theorem 11.3]{AHMMS} is not stated in exactly the form as given above. The volume polynomial as discussed in \cite[Section 11]{AHMMS} is the polynomial whose annihilator is equal to the ideal defining the cohomology ring $H^*(\Hess(S,h))$ of the regular semisimple Hessenberg variety; it is only determined up to a scalar multiple. In particular, the polynomial $P_I$ as written in \cite[Section 11]{AHMMS} does not include the scalar multiple $\frac{1}{1!2!\cdots (n-1)!}$ given in the formula~\eqref{eq: vol flag}. The version used in the proof of \cite[Theorem 6.2]{ADGH}, on the other hand, does include this scalar multiple; its proof was based on the ideas of \cite[Theorem 11.3]{AHMMS}. 
\end{remark}

It immediately follows from Theorem~\ref{theorem: AHMMS} and~\eqref{eq: vol Flag is vol GZ} that \begin{equation}\label{eq: vol Hess as derivative}
\vol_\lambda(\Hess(S,h)) = \left( \prod_{j=1}^{n-1} \prod_{i=h(j)+1}^n (\partial_j - \partial_i) \right) \vol(\GZ(\lambda)).
\end{equation}

From~\eqref{eq: vol GZ} we can see that~\eqref{eq: vol Hess as derivative} expresses $\vol_\lambda(\Hess(S,h))$ as a result of a sequence of partial derivatives applied to a polynomial which has non-negative coefficients when written in the basis of monomials in the $\alpha_i=\lambda_i-\lambda_{i+1}$. 
Thus, it was natural for us to ask the following questions. 

\begin{itemize}
\item Can we explicitly compute the RHS of~\eqref{eq: vol Hess as derivative} in terms of the (volumes of the) faces of $\GZ(\lambda)$? 
\item Is the RHS of~\eqref{eq: vol Hess as derivative} a non-negative linear combination of monomials in the $\alpha_i:=\lambda_i-\lambda_{i+1}$'s? 
\end{itemize} 

The formula in Corollary~\ref{cor: Hess AT KST} will provide an answer to the first question, and a generalization of a result of Postnikov will answer the second. We discuss this in more detail in the next sections.

Before proceeding, however, we take a moment to note that the partial derivative operators appearing in~\eqref{eq: vol Hess as derivative} have a visual interpretation in terms of the box diagram corresponding to $h$ as in Figure~\ref{picture:h=(3,3,4,5,5)}. Specifically, the $(i,j)$ for which the operator $\partial_i - \partial_j$ appears in~\eqref{eq: vol Hess as derivative} correspond precisely to the boxes which are \emph{not} colored in the box diagram of $h$. 

\begin{example} 
Continuing with the setup of Example~\ref{example: box diagram}, the boxes that are not colored in Figure~\ref{picture:h=(3,3,4,5,5)} are the ones corresponding to the indices $(4,1), (4,2), (5,1), (5,2),$ and $(5,3)$. Accordingly it can be seen from~\eqref{eq: vol Hess from vol Fl} that in this case we have
\[
\vol_\lambda(\Hess(S,h)) = (\partial_1-\partial_4)(\partial_1-\partial_5)(\partial_2-\partial_4)(\partial_2-\partial_5)(\partial_3-\partial_5) \vol_\lambda(\Fl(\C^5)).
\]  
\end{example}

Moreover, it turns out that the results of the partial derivative operations in the RHS of~\eqref{eq: vol Hess as derivative} can, at least in small-$n$ cases, be interpreted very concretely in terms of the faces of $\GZ(\lambda)$. We illustrate some examples below; these gave us the idea for this paper.

\begin{example} 
Let $n=3$ and consider the Hessenberg function $h=(2,3,3)$. In this case~\eqref{eq: vol Hess as derivative} implies that we have 
\[
\vol_\lambda(\Hess(S,h)) = (\partial_1 - \partial_3) \vol(\GZ(\lambda)) = 
\partial_1 \vol(\GZ(\lambda)) - \partial_3 \vol(\GZ(\lambda)).
\]
The above formula can be understood directly and geometrically in terms of the polytope, as follows. Note that for $\varepsilon >0$ sufficiently small and for $\lambda' = (\lambda_1+\varepsilon, \lambda_2, \lambda_3)$, the difference between the volume of the polytope $\GZ(\lambda')$ and $\GZ(\lambda)$ is the volume (area) of the facet of $\GZ(\lambda)$ specified by $x=\lambda_1$ multiplied by $\varepsilon$. Thus, $\partial_1(\GZ(\lambda))$ is the area of the facet $\{x = \lambda_1\}$. A similar straightforward argument yields that $- \partial_3(\GZ(\lambda))$ is the area of the facet $\{y=\lambda_3\}$. Thus, we obtain that $\vol_\lambda(\Hess(S,h))$ is equal to the sum of the areas of the two facets $\{x=\lambda_1\}$ and $\{y=\lambda_3\}$, which are two of the six facets illustrated in Figure~\ref{fig: facets and dots for n=3}. 
\end{example} 

\begin{example}\label{example: n=4 first ex} 
Let $n=4$ and $\lambda_1 > \lambda_2 > \lambda_3 > \lambda_4$. Then the coordinates are organized as follows: 
\begin{equation}\label{eq:GZ n=4}
    \begin{array}{cccc}
        \lambda_1&  x_{1,2}    &    x_{1,3} & x_{1,4} \\   
         & \lambda_2 & x_{2,3} & x_{2,4}  \\
          & & \lambda_3 & x_{3,4} \\ 
           &  &  & \lambda_4. \\ 
     \end{array}
\end{equation}
For the Hessenberg function $h=(3,4,4,4)$, an argument similar to Example~\ref{example: volume for n=3 h=233} yields that $\vol_\lambda(\Hess(S,h))$ is the sum of the areas of the facets $\{x_{1,2}=\lambda_1\}$ and $\{x_{3,4}=\lambda_4\}$. See Figure~\ref{fig:hess_vol n=4 h=3444}. 

\begin{figure}[h]
\begin{center}
\begin{subfigure}[b]{.24\textwidth}
\begin{center}
    \begin{tikzpicture}[scale=0.6,font=\footnotesize]
    \tikzset{
    solid node/.style={circle,draw,inner sep=1.2,fill=black}
    }
        \node [solid node](01) at (0,3) {};
        \node [solid node](02) at (1,2) {};
        \node [solid node](03) at (2,1) {};
        \node [solid node](04) at (3,0) {};
        \node [solid node](11) at (1,3) {};
        \node [solid node](12) at (2,2) {};
        \node [solid node](13) at (3,1) {};
        \node [solid node](21) at (2,3) {};
        \node [solid node](22) at (3,2) {};
        \node [solid node](31) at (3,3) {};
        \path[thick] (01) edge (11);
    \end{tikzpicture}
    \end{center}
        \subcaption*{$\lambda_1=x_{1,2}$}
    \end{subfigure}
    \begin{subfigure}[b]{.24\textwidth}
        \begin{center}
    \begin{tikzpicture}[scale=0.6,font=\footnotesize]
    \tikzset{
    solid node/.style={circle,draw,inner sep=1.2,fill=black}
    }
        \node [solid node](01) at (0,3) {};
        \node [solid node](02) at (1,2) {};
        \node [solid node](03) at (2,1) {};
        \node [solid node](04) at (3,0) {};
        \node [solid node](11) at (1,3) {};
        \node [solid node](12) at (2,2) {};
        \node [solid node](13) at (3,1) {};
        \node [solid node](21) at (2,3) {};
        \node [solid node](22) at (3,2) {};
        \node [solid node](31) at (3,3) {};
        \path[thick] (04) edge (13);
        \end{tikzpicture}
        \end{center}
        \subcaption*{$x_{3,4}=\lambda_4$}
    \end{subfigure}
    \end{center}
    \caption{Face diagrams corresponding to $\vol_\lambda(\Hess(S,h))$ when $h=(3,4,4,4)$}\label{fig:hess_vol n=4 h=3444}
\end{figure}
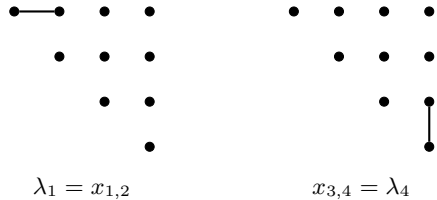

\end{example}

We take a moment to note that the results of Section~\ref{sec: combinatorics} additionally allows us to explicitly and directly interpret the results of the partial derivative operations in the RHS of~\eqref{eq: vol Hess as derivative} in terms of the faces of $\GZ(\lambda)$, in a manner independent of Theorem~\ref{thm: vol Hess 1}.

\begin{example}\label{example: n=4 h=2444 volume}
Let $n=4$ and suppose $h = (2,4,4,4)$. For the purpose of the next computation we denote by $S_1$ and $S_2$ the volumes of the areas of the facets $\{x_{1,2}=\lambda_1\}$ and $\{x_{3,4}=\lambda_4\}$ respectively. From~\eqref{eq: vol Hess as derivative} and from the computation in Example~\ref{example: n=4 first ex} we can conclude  
\begin{equation*}
\begin{split} 
\vol_\lambda(\Hess(S,h)) & = (\partial_1 - \partial_3) \vol_\lambda(\Hess(S,(3,4,4,4))) \\
 & = \partial_1(S_1) + \partial_1(S_2) - \partial_3(S_1) - \partial_3(S_2).
\end{split} 
\end{equation*}
Again, an argument similar to Example~\ref{example: volume for n=3 h=233} shows that $\partial_1(S_1)$ gives the area of the codimension-$2$ face $\{x_{1,2}=x_{1,3}=\lambda_1\}$ and $\partial_1(S_2)$ gives the area of $\{x_{1,2}=\lambda_1, x_{3,4}=\lambda_4\}$. In the same way we can compute that $-\partial_3(S_2)$ is the volume of the face $\{x_{2,3}=\lambda_3, x_{3,4}=\lambda_4\}$. The description of the term $- \partial_3(S_1)$ is slightly more involved. First, an argument similar to the previous ones shows that $-\partial_3(S_1)$ is the difference 
\[
- \partial_3(S_1) = \vol(\{x_{1,2}=\lambda_1, x_{2,3} = \lambda_3\}) - \vol(\{x_{1,2}=\lambda_1, x_{3,4}=\lambda_3\}).
\]
Now, applying Proposition~\ref{prop:x-relation} we obtain the relation 
\[
- \vol(\{x_{1,2}=\lambda_1, x_{3,4}=\lambda_3\}) = \vol(\{x_{1,2}=\lambda_1, x_{3,4}=x_{2,4}\}) -
\vol(\{x_{1,2}=\lambda_1, x_{3,4}=\lambda_4\}).
\]
Combining the above equations, we conclude that $\vol_\lambda(\Hess(S,h))$ is the sum of the volumes of the four faces of $\GZ(\lambda)$ depicted in Figure~\ref{fig: n=4 h=2444}, which is consistent with what we obtained in Example~\ref{example: cohomology n=4 h=2444}. 
\end{example}

We can now answer the two questions posed earlier. 
Firstly, from Corollary~\ref{cor: Hess AT KST} we obtain the following. 

\begin{theorem}\label{thm: vol Hess 1} 
Let $\Hess(S,h)$ be a regular semisimple Hessenberg variety. 
Then 
\begin{equation*}
\hspace{-30pt}
\vol_\lambda(\Hess(S,h))=\sum_{\substack{u,v \in \mathfrak{S}_n \\ v^{-1}u=\wh \\ \ell(u)+\ell(v)=\ell(\wh)}} \sum_{\substack{F:\text{ reduced Kogan face} \\ F^*:\text{ reduced dual Kogan face} \\ w(F)=u,\ w(F^*)=v}}  \vol(F \cap F^*)
\end{equation*}
where $\vol(F \cap F^*)$ denotes the $(m-d)$-dimensional volume  of $F\cap F^\ast$, where $m=n(n-1)/2$ and $d=\ell(u)+\ell(v)$. 
\end{theorem} 

\begin{proof} 
It follows from Theorem~\ref{theorem: KST} and the argument given in \cite[Proof of Theorem 4.3]{KST12} that 
\begin{equation}\label{eq: KST for volumes}
\vol_{\lambda}([X^u][X^{w_0vw_0}])=   \sum_{\substack{F:\text{ reduced Kogan face} \\ F^*:\text{ reduced dual Kogan face} \\ w(F)=u,\ w(F^*)=v}} \vol(F\cap F^*). 
\end{equation}
Since $\vol_{\lambda}$ is linear on cohomology classes by definition, the claim now follows immediately from Theorem~\ref{theorem:AT} and~\eqref{eq: KST for volumes}. 
\end{proof}

The above theorem provides an answer to the first question posed above, namely, we have expressed the RHS of~\eqref{eq: vol Hess as derivative} explicitly in terms of volumes of faces of $\GZ(\lambda)$.

Next, we answer the second question posed above. 
Indeed, by applying Proposition~\ref{prop:face of GZ}, we also immediately obtain the following, which is a manifestly positive and combinatorial formula for the volume polynomial of regular semisimple Hessenberg varieties, expressed in terms of $\alpha_i$'s.

\begin{theorem}\label{thm:vol_Hess 2}
Let $\Hess(S,h)$ be a regular semisimple Hessenberg variety. Then 
\begin{equation*}
\hspace{-30pt}
\vol_\lambda(\Hess(S,h))=\sum_{\substack{u,v \in \mathfrak{S}_n \\ v^{-1}u=\wh \\ \ell(u)+\ell(v)=\ell(\wh)}} \sum_{\substack{F:\text{ reduced Kogan face} \\ F^*:\text{ reduced dual Kogan face} \\ w(F)=u,\ w(F^*)=v}}\sum_{p_1,\dots,p_{n-1}\ge 0}N_{F \cap F^*}(p_1,\dots,p_{n-1})\frac{\alpha_1^{p_1}}{p_1!}\cdots\frac{\alpha_{n-1}^{p_{n-1}}}{p_{n-1}!}
\end{equation*}
where $N_{F \cap F^*}(p_1,\dots,p_{n-1})$ is the number of shifted tableaux $T$ associated with $F \cap F^*$ with the diagonal vector $\mathrm{diag}(T)=(1,p_1+2,p_1+p_2+3,\ldots,p_1+\cdots+p_{n-1}+n)$.
\end{theorem}

We illustrate the formula above in a simple example. 

\begin{example}
Continuing with the setup of Example~\ref{example: n=4 h=2444 volume}, 
let $F_1,\ldots,F_4$ denote the four faces (from left to right) in Figure~\ref{fig: n=4 h=2444}. Then we can compute the volume of each of the faces $F_1,\ldots,F_4$ by Proposition~\ref{prop:face of GZ} and we obtain: 
\begin{equation*}
\begin{split}
	\vol(F_1)&= \alpha_1\frac{\alpha_2^2}{2!}\alpha_3+2\frac{\alpha_2^3}{3!}\alpha_3+\alpha_1\alpha_2\frac{\alpha_3^2}{2!} + 2 \frac{\alpha_2^2}{2!}\frac{\alpha_3^2}{2!} + \alpha_2\frac{\alpha_3^3}{3!}\\
   \vol(F_2)&=\frac{\alpha_1^2}{2!}\frac{\alpha_3^2}{2!} + \alpha_1\alpha_2\frac{\alpha_3^2}{2!} + \frac{\alpha_2^2}{2!}\frac{\alpha_3^2}{2!}+\alpha_1\frac{\alpha_3^3}{3!}+\alpha_2\frac{\alpha_3^3}{3!}\\
	\vol(F_3)&=\frac{\alpha_1^2}{2!}\alpha_2\alpha_3+2\alpha_1\frac{\alpha_2^2}{2!}{\alpha_3}+2\frac{\alpha_2^3}{3!}\alpha_3+\alpha_1\alpha_2\frac{\alpha_3^2}{2!}+\frac{\alpha_2^2}{2!}\frac{\alpha_3^2}{2!}\\
	\vol(F_4)&=\frac{\alpha_1^3}{3!}\alpha_3+\frac{\alpha_1^2}{2!}\alpha_2\alpha_3+\alpha_1\frac{\alpha_2^2}{2!}\alpha_3+\frac{\alpha_1^2}{2!}\frac{\alpha_3^2}{2!}+\alpha_1\alpha_2\frac{\alpha_3^2}{2!}\\
\end{split}
\end{equation*}
Hence, when $n=4$ and $h=(2,4,4,4)$, the volume of $\Hess(S,h)$ is
$$
\frac{\alpha_1^3}{3!}\alpha_3+2\frac{\alpha_1^2}{2!}\alpha_2\alpha_3+4\alpha_1\frac{\alpha_2^2}{2!}\alpha_3+4\frac{\alpha_2^3}{3!}\alpha_3+2\frac{\alpha_1^2}{2!}\frac{\alpha_3^2}{2!}+4\alpha_1\alpha_2\frac{\alpha_3^2}{2!} +4 \frac{\alpha_2^2}{2!}\frac{\alpha_3^2}{2!}+\alpha_1\frac{\alpha_3^3}{3!}+2\alpha_2\frac{\alpha_3^3}{3!}.
$$
\end{example}

\section{A decomposition of the permutohedron into cubes} \label{sec:decomposing}

In this section we prove a combinatorial result about the permutohedron, which is of independent interest, but which were inspired by considerations in the previous sections. Let $\lambda = (\lambda_1, \lambda_2, \cdots, \lambda_n)$ where $\lambda_1 > \lambda_2 > \cdots > \lambda_n$, and let $\mathrm{Perm}(\lambda)$ denote the polytope obtained as the convex hull of the $n!$ vertices in $\R^n$ obtained by permuting the entries of $\lambda$. The polytope $\mathrm{Perm}(\lambda)$ is $(n-1)$-dimensional, and is the moment map image of the permutohedral variety $\Hess(S, h_1)$ where $h_1 := (2,3,\cdots, n, n)$. In the arguments below, we use the relationship between the Gelfand-Zetlin polytope $\GZ(\lambda)$ and $\mathrm{Perm}(\lambda)$ in order to show that the polytope $\mathrm{Perm}(\lambda)$ can be decomposed into $(n-1)!$ many subpolytopes, each of which is combinatorially an $(n-1)$-cube. 

We first recall that the one-to-one correspondence between the set of simple vertices in $\GZ(\lambda)$ and the permutation group $\mathfrak{S}_n$ \cite[Section 5]{Ki2010}. For each integer $i$ with $0 \leq i \leq n-1$, choose an integer $d_i$ satisfying $1 \leq d_i \leq n-i$. There are $n!$ many possible such sequences $d = (d_0, d_1, \ldots, d_{n-1})$. Moreover, each such $d$ corresponds to a set of equalities 
\begin{equation}\label{eq: perm eq}
x_{i, i+j} = x_{i, i+j+1} \textup{ for } 1 \leq i < d_j, \quad \textup{ and } \quad 
x_{i, i+j} = x_{i-1, i+j} \textup{ for } d_j < i \leq n-j.
\end{equation}
The above set of equalities specifies a single simple vertex of $\GZ(\lambda)$, which we denote $v_d$. Conversely, it is known that any simple vertex of $\GZ(\lambda)$ is $v_d$ for some $d$ as above. It is also useful to visualize this in terms of the face diagrams as in Figure~\ref{fig:ex:face_diagram}. Indeed, the face diagram for a set of equations of the form~\eqref{eq: perm eq} is a collection of $n$ paths (possibly of length $0$). For each integer $k$ with $1 \leq k \leq n$, let $w_d(k)$ denote the number of vertices in the face diagram contained in the path which contains the vertex corresponding to $\lambda_k$.

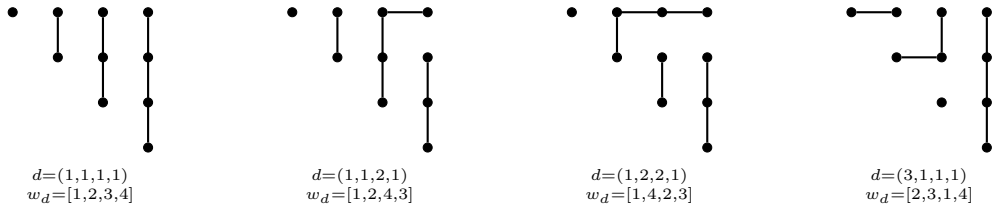
\begin{figure}[h]
\centering
\begin{subfigure}[b]{.24\textwidth}
\centering
\begin{tikzpicture}[scale=0.6,font=\footnotesize]
    \tikzset{
    solid node/.style={circle,draw,inner sep=1.2,fill=black}
    }
        \node [solid node](01) at (0,3) {};
        \node [solid node](02) at (1,2) {};
        \node [solid node](03) at (2,1) {};
        \node [solid node](04) at (3,0) {};
        \node [solid node](11) at (1,3) {};
        \node [solid node](12) at (2,2) {};
        \node [solid node](13) at (3,1) {};
        \node [solid node](21) at (2,3) {};
        \node [solid node](22) at (3,2) {};
        \node [solid node](31) at (3,3) {};
        \path[thick] (02) edge (11);
        \path[thick](03) edge (21);
        \path[thick](04) edge (31);
    \end{tikzpicture}
        \subcaption*{$d=(1,1,1,1)\atop w_d=[1,2,3,4]$}
\end{subfigure}
\begin{subfigure}[b]{.24\textwidth}
\centering
\begin{tikzpicture}[scale=0.6,font=\footnotesize]
    \tikzset{
    solid node/.style={circle,draw,inner sep=1.2,fill=black}
    }
        \node [solid node](01) at (0,3) {};
        \node [solid node](02) at (1,2) {};
        \node [solid node](03) at (2,1) {};
        \node [solid node](04) at (3,0) {};
        \node [solid node](11) at (1,3) {};
        \node [solid node](12) at (2,2) {};
        \node [solid node](13) at (3,1) {};
        \node [solid node](21) at (2,3) {};
        \node [solid node](22) at (3,2) {};
        \node [solid node](31) at (3,3) {};
        \path[thick] (02) edge (11);
        \path[thick](03) edge (21);
        \path[thick] (21)  edge (31);
        \path[thick](04) edge (22);
    \end{tikzpicture}
        \subcaption*{$d=(1,1,2,1) \atop w_d=[1,2,4,3]$}
\end{subfigure}
\begin{subfigure}[b]{.24\textwidth}
\centering
\begin{tikzpicture}[scale=0.6,font=\footnotesize]
    \tikzset{
    solid node/.style={circle,draw,inner sep=1.2,fill=black}
    }
        \node [solid node](01) at (0,3) {};
        \node [solid node](02) at (1,2) {};
        \node [solid node](03) at (2,1) {};
        \node [solid node](04) at (3,0) {};
        \node [solid node](11) at (1,3) {};
        \node [solid node](12) at (2,2) {};
        \node [solid node](13) at (3,1) {};
        \node [solid node](21) at (2,3) {};
        \node [solid node](22) at (3,2) {};
        \node [solid node](31) at (3,3) {};
        \path[thick] (02) edge (11);
        \path[thick] (11) edge (31);
        \path[thick](03) edge (12);
        \path[thick](04) edge (22);
    \end{tikzpicture}
        \subcaption*{$d=(1,2,2,1) \atop w_d=[1,4,2,3]$}
\end{subfigure}
\begin{subfigure}[b]{.24\textwidth}
\centering
\begin{tikzpicture}[scale=0.6,font=\footnotesize]
    \tikzset{
    solid node/.style={circle,draw,inner sep=1.2,fill=black}
    }
        \node [solid node](01) at (0,3) {};
        \node [solid node](02) at (1,2) {};
        \node [solid node](03) at (2,1) {};
        \node [solid node](04) at (3,0) {};
        \node [solid node](11) at (1,3) {};
        \node [solid node](12) at (2,2) {};
        \node [solid node](13) at (3,1) {};
        \node [solid node](21) at (2,3) {};
        \node [solid node](22) at (3,2) {};
        \node [solid node](31) at (3,3) {};
        \path[thick] (01) edge (11);
        \path[thick] (02) edge (12);
        \path[thick] (12) edge (21);
        \path[thick](04) edge (31);
    \end{tikzpicture}
        \subcaption*{$d=(3,1,1,1) \atop w_d=[2,3,1,4]$}
\end{subfigure}
\caption{Examples of $v_d$ and $w_d$ when $n=4$}
\end{figure}


Then $w_d := (w_d(k))_{k=1}^n$ defines a permutation in $\mathfrak{S}_n$, and the correspondence $v_d \mapsto w_d$ is the bijection between the set of simple vertices of $\GZ(\lambda)$ and $\mathfrak{S}_n$.

Next, we briefly recall the relation between $\GZ(\lambda)$ and $\Perm(\lambda)$. For a point $(x_{i,j}) \in \GZ(\lambda) \subseteq \R^{n(n-1)/2}$ let us define 
\begin{equation}\label{eq: def y_k}
y_k := \sum_{i=1}^{n-k} x_{i, i+k} 
\end{equation}
for $k$ an integer with $0 \leq k \leq n-1$. Since $x_{i,i} = \lambda_i$ by definition, we conclude that $y_0 := \sum_{i=1}^n \lambda_i$ is a constant.  Define the map 
\begin{equation}\label{eq: def Phi}
\Phi((x_{i,j})) := (y_0 -y_1, y_1 - y_2, \cdots, y_{n-2}-y_{n-1}, y_{n-1})
\end{equation}
from $\R^{n(n-1)/2} \to \R^n$. Then it follows that $\Phi(\GZ(\lambda)) = \Perm(\lambda)$ \cite[Section 5]{KST12}.  It is straightforward to check that if we let $p_w$ denote the simple vertex of $\GZ(\lambda)$ corresponding to the permutation $w \in \mathfrak{S}_n$, then $\Phi(p_w) = (\lambda_{w^{-1}(1)}, \lambda_{w^{-1}(2)}, \cdots, \lambda_{w^{-1}(n)})$.

The following technical lemma will be useful in what follows.

\begin{lemma}\label{lemma: move b} 
Let $a_1 > a_2 > \cdots > a_m$ be real numbers. Suppose $b_1$, $b_2$, $\cdots$, $b_{m-1}$ are real numbers satisfying the inequalities 
\begin{equation*} \label{eq:3-1}
    \begin{array}{cccccc}
        a_1&  b_1    &       &         & 	& \\
                         &a_2&  b_2       &      & 	& \\
                         &                 &  \ddots	      &		\ddots    &	      &\\
                         &                 &                  &a_{m-2}&b_{m-2}	 &\\
                         &                 &                  &                      &a_{m-1}&b_{m-1}\\
                         &                 &                  &                      &                      &a_m
    \end{array}
\end{equation*} 
where we take the convention of~\eqref{eq:abc}. Then there exists an integer $k$ with $1 \leq k \leq m-1$ and a unique real number $b'_k$ with $a_k \geq b'_k \geq a_{k+1}$ such that 
for the collection of real numbers $b'_j$ ($j \neq k$) defined by 
\[
b'_1 := a_1, \cdots, b'_{k-1} := a_{k-1}, b'_{k+1} := a_{k+2}, \cdots, b'_{m-1} = a_m
\]
we have 
\begin{equation}\label{eq: b' and b} 
\sum_{j=1}^{m-1} b_j = \sum_{j=1}^{m-1} b'_j. 
\end{equation}

\end{lemma}

\begin{proof} 
It is clear from the definition of the $b'_j$ and the property~\eqref{eq: b' and b} that if such real numbers $b'_j$ exist then the choice is unique. So it suffices to prove existence. Set $A := \sum_{j=1}^m a_j$ and $B := \sum_{j=1}^{m-1} b_j$. Then from the inequality $a_{i+1} \leq b_i$ we obtain $A-a_1 \leq B$. From the inequality $b_i \leq a_i$ we obtain $B \leq A - a_m$. Moreover, from the assumption $a_1 > a_2 > \cdots > a_m$ we obtain 
\[
A- a_1 < A- a_2 < \cdots < A - a_m.
\]
Therefore, there must exist a $k$ such that $A - a_k \leq B \leq A-a_{k+1}$. Define $b'_k := B - (A-a_k-a_{k+1})$ for this $k$, and define $b'_j$ for $j \neq k$ as given in the claim. Then $a_k \geq b'_k \geq a_{k+1}$ and $\sum_{j=1}^{m-1} b_j = \sum_{j=1}^{m-1} b'_j$. The claim follows. 
\end{proof}

In order to state the next result, we introduce some notation for certain faces of $\GZ(\lambda)$. First we define 
\begin{equation}\label{eq: def calFn}
\mathcal{F}_n := \{ \mathbf{r} = (r_1, \ldots, r_{n-1}) \in \Z^{n-1} \, \mid \, 1 \leq r_j \leq n-j, \textup{ for } 1 \leq j \leq n-1\}.
\end{equation}
For an element $\mathbf{r} = (r_1, \ldots, r_{n-1}) \in \mathcal{F}_n$, we let $F(\mathbf{r})$ denote the face of $\GZ(\lambda)$ defined by the $(n-1)(n-2)/2$ many equalities 
\begin{equation}\label{eq: equations for Fr}
x_{i, j+i} = x_{i, j+i-1} \textup{ for } 1 \leq i < r_j, \quad \textup{ and } \quad x_{i, j+i} = x_{i+1, j+i} \textup{ for } r_j < i \leq n-j.
\end{equation}
(Note that $r_{n-1}=1$ by definition so when $j=n-1$ the corresponding equation in~\eqref{eq: equations for Fr} is vacuous.) In Figure~\ref{fig:example-F(r)} we illustrate some examples.

\begin{figure}[h]
    \begin{center}
    \begin{subfigure}[b]{.3\textwidth}
    \begin{center}
    \begin{tikzpicture}[scale=0.6,font=\footnotesize]
        \tikzset{
        solid node/.style={circle,draw,inner sep=1.2,fill=black},
        hollow node/.style={circle,draw,inner sep=1.2}
        }
        \node [solid node](01) at (0,3) {};
        \node [solid node](02) at (1,2) {};
        \node [solid node](03) at (2,1) {};
        \node [solid node](04) at (3,0) {};
        \node [solid node](11) at (1,3) {};
        \node [solid node](12) at (2,2) {};
        \node [solid node](13) at (3,1) {};
        \node [solid node](21) at (2,3) {};
        \node [solid node](22) at (3,2) {};
        \node [solid node](31) at (3,3) {};
        \path (12) edge (03);
        \path (13) edge (04);
        \path (22) edge (13);
    \end{tikzpicture}
    \end{center}
    \subcaption*{$F(1,1,1)$}
    \end{subfigure}
    \begin{subfigure}[b]{.3\textwidth}
    \begin{center}
    \begin{tikzpicture}[scale=0.6,font=\footnotesize]
        \tikzset{
        solid node/.style={circle,draw,inner sep=1.2,fill=black},
        hollow node/.style={circle,draw,inner sep=1.2}
        }
        \node [solid node](01) at (0,3) {};
        \node [solid node](02) at (1,2) {};
        \node [solid node](03) at (2,1) {};
        \node [solid node](04) at (3,0) {};
        \node [solid node](11) at (1,3) {};
        \node [solid node](12) at (2,2) {};
        \node [solid node](13) at (3,1) {};
        \node [solid node](21) at (2,3) {};
        \node [solid node](22) at (3,2) {};
        \node [solid node](31) at (3,3) {};
        \path (11) edge (01);
        \path (13) edge (04);
        \path (22) edge (13);
    \end{tikzpicture}
    \end{center}
    \subcaption*{$F(2,1,1)$}
    \end{subfigure}
    \begin{subfigure}[b]{.3\textwidth}
    \begin{center}
    \begin{tikzpicture}[scale=0.6,font=\footnotesize]
        \tikzset{
        solid node/.style={circle,draw,inner sep=1.2,fill=black},
        hollow node/.style={circle,draw,inner sep=1.2}
        }
        \node [solid node](01) at (0,3) {};
        \node [solid node](02) at (1,2) {};
        \node [solid node](03) at (2,1) {};
        \node [solid node](04) at (3,0) {};
        \node [solid node](11) at (1,3) {};
        \node [solid node](12) at (2,2) {};
        \node [solid node](13) at (3,1) {};
        \node [solid node](21) at (2,3) {};
        \node [solid node](22) at (3,2) {};
        \node [solid node](31) at (3,3) {};
        \path (11) edge (01);
        \path (12) edge (02);
        \path (22) edge (13);
    \end{tikzpicture}
    \end{center}
    \subcaption*{$F(3,1,1)$}
    \end{subfigure}

    \vspace{.2cm}

    \begin{subfigure}[b]{.3\textwidth}
    \begin{center}
    \begin{tikzpicture}[scale=0.6,font=\footnotesize]
        \tikzset{
        solid node/.style={circle,draw,inner sep=1.2,fill=black},
        hollow node/.style={circle,draw,inner sep=1.2}
        }
        \node [solid node](01) at (0,3) {};
        \node [solid node](02) at (1,2) {};
        \node [solid node](03) at (2,1) {};
        \node [solid node](04) at (3,0) {};
        \node [solid node](11) at (1,3) {};
        \node [solid node](12) at (2,2) {};
        \node [solid node](13) at (3,1) {};
        \node [solid node](21) at (2,3) {};
        \node [solid node](22) at (3,2) {};
        \node [solid node](31) at (3,3) {};
        \path (12) edge (03);
        \path (13) edge (04);
        \path (21) edge (11);
    \end{tikzpicture}
    \end{center}
    \subcaption*{$F(1,2,1)$}
    \end{subfigure}
    \begin{subfigure}[b]{.3\textwidth}
    \begin{center}
    \begin{tikzpicture}[scale=0.6,font=\footnotesize]
        \tikzset{
        solid node/.style={circle,draw,inner sep=1.2,fill=black},
        hollow node/.style={circle,draw,inner sep=1.2}
        }
        \node [solid node](01) at (0,3) {};
        \node [solid node](02) at (1,2) {};
        \node [solid node](03) at (2,1) {};
        \node [solid node](04) at (3,0) {};
        \node [solid node](11) at (1,3) {};
        \node [solid node](12) at (2,2) {};
        \node [solid node](13) at (3,1) {};
        \node [solid node](21) at (2,3) {};
        \node [solid node](22) at (3,2) {};
        \node [solid node](31) at (3,3) {};
        \path (11) edge (01);
        \path (13) edge (04);
        \path (21) edge (11);
    \end{tikzpicture}
    \end{center}
    \subcaption*{$F(2,2,1)$}
    \end{subfigure}
    \begin{subfigure}[b]{.3\textwidth}
    \begin{center}
    \begin{tikzpicture}[scale=0.6,font=\footnotesize]
        \tikzset{
        solid node/.style={circle,draw,inner sep=1.2,fill=black},
        hollow node/.style={circle,draw,inner sep=1.2}
        }
        \node [solid node](01) at (0,3) {};
        \node [solid node](02) at (1,2) {};
        \node [solid node](03) at (2,1) {};
        \node [solid node](04) at (3,0) {};
        \node [solid node](11) at (1,3) {};
        \node [solid node](12) at (2,2) {};
        \node [solid node](13) at (3,1) {};
        \node [solid node](21) at (2,3) {};
        \node [solid node](22) at (3,2) {};
        \node [solid node](31) at (3,3) {}; 
        \path (11) edge (01);
        \path (12) edge (02);
        \path (21) edge (11);
    \end{tikzpicture}
    \end{center}
    \subcaption*{$F(3,2,1)$}
    \end{subfigure}
    \end{center}
    \caption{Example for $n=4$}\label{fig:example-F(r)}
\end{figure}
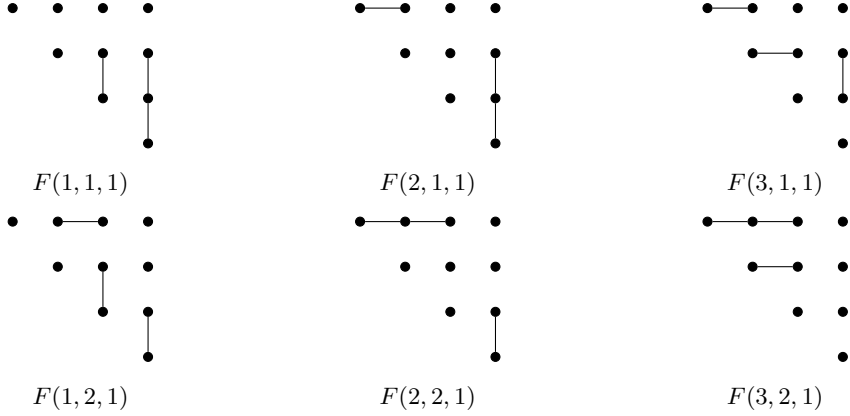

We have the following. 

\begin{proposition}\label{prop: volume-preserving}
Let $\Phi\colon \bigcup_{\mathbf{r}\in\mathcal{F}_n}F(\mathbf{r})\to \Perm(\lambda)$ be the restriction to $\bigcup_{\mathbf{r} \in \mathcal{F}_n} F(\mathbf{r}) \subseteq \GZ(\lambda)$ of the map $\Phi$ defined in~\eqref{eq: def Phi}. Then $\Phi$ is a bijection between $\bigcup_{\mathbf{r}\in\mathcal{F}_n}F(\mathbf{r})$ and $\Perm(\lambda)$, and $\Phi$ is volume-preserving  on each face $F(\mathbf{r})$. In particular, $\vol(\Perm(\lambda))=\sum_{\mathbf{r}\in \mathcal{F}_n}\vol(F(\mathbf{r}))$.
\end{proposition} 

\begin{proof} 
Let $(x_{i,j}) \in \GZ(\lambda)$. Define 
\begin{equation}\label{eq: def Psi} 
\Psi((x_{i,j})) := (y_0, y_1, \ldots, y_{n-1}) \in \R^n. 
\end{equation}
It is then clear from~\eqref{eq: def Phi} that $\Phi$ and $\Psi$ are related by a matrix multiplication, i.e. 
\[
\Phi((x_{i,j})) = \Psi((x_{i,j})) A \quad \textup{ for some } A \in \SL_n(\Z).
\]
Since we know $\Phi(\GZ(\lambda)) = \Perm(\lambda)$, to prove the first claim of the proposition it suffices to show that the restriction $\overline{\Psi}$ of $\Psi$ to $\bigcup_{\mathbf{r} \in \mathcal{F}_n} F(\mathbf{r})$ 
\[
\overline{\Psi}: \bigcup_{\mathbf{r} \in \mathcal{F}_n} F(\mathbf{r}) \to \Psi(\GZ(\lambda)) \subseteq \R^n
\]
is a bijection.  We do this using Lemma~\ref{lemma: move b} repeatedly, as follows.  

We start with surjectivity. Suppose given a point $\Psi((x_{i,j}))$ for $(x_{i,j}) \in \GZ(\lambda)$. We construct a point $(x'_{i,j}) \in \bigcup_{\mathbf{r} \in \mathcal{F}_n} F(\mathbf{r})$ which lies in the preimage of $(x_{i,j})$ as follows, by repeatedly applying Lemma~\ref{lemma: move b}. First, take $a_i=\lambda_i$ for $1 \leq i \leq n$ and $b_i = x_{i,i+1}$ for $1 \leq i \leq n-1$ and apply Lemma~\ref{lemma: move b} to obtain values $b'_i$ satisfying $\sum b'_i = \sum x_{i,i+1} = y_1$. Now set $x'_{i,i+1} := b'_i$ for $1 \leq i \leq n-1$. Next, redefine $a_i := x'_{i,i+1}$ for $1 \leq i \leq n-1$ and $b_i := x_{i, i+2}$ for $1 \leq i \leq n-2$ and apply Lemma~\ref{lemma: move b} again to obtain (new) values $b'_i$ with $\sum b'_i = \sum x_{i, i+2} = y_2$. Set $x'_{i, i+2} := b'_i$ for $1 \leq i \leq n-2$. Proceeding similarly we obtain $(x'_{i,j})$ which, by construction, maps to $\Psi((x_{i,j})) = (y_1, y_2, \ldots, y_{n-1})$ and which lies in a face $F(\mathbf{r})$ of $\GZ(\lambda)$ for some $\mathbf{r}$. Hence, $\overline{\Psi}$ is surjective. It follows from the second claim of Lemma~\ref{lemma: move b} that $\overline{\Psi}$ is injective. 

Next we prove the second claim of the proposition. From the fact that $\Phi=\Psi \cdot A$ for $A \in \SL_n(\Z)$ it follows that it suffices to show that $\overline{\Psi}$ preserves volume when restricted to each $F(\mathbf{r})$. For the remainder of this argument we may therefore fix an $\mathbf{r} \in \mathcal{F}_n$. For each such $\mathbf{r}$, the map $(x_{i,j}) \mapsto (x_{r_j, j+r_j})_{1 \leq j \leq n-1}$ naturally identifies the $(n-1)$-dimensional subspace spanned by $F(\mathbf{r})$ with $\R^{n-1}$. In what follows, we take as coordinates in $\R^{n-1}$ the variables $z_j := x_{r_j, j+r_j}$ and identify $F(\mathbf{r})$ with its image in $\R^{n-1}$ under the map above. With respect to these coordinates we have 
\[
\Psi(z_1,\dots,z_{n-1})=(z_1,\dots,z_{n-1})B \quad (B\in \SL_{n-1}(\Z)).
\]
Here we have dropped the coordinate $y_0 = \sum_{j=1}^n \lambda_j$ since it is just a constant and we are ignoring the parallel translation by $y_0$ since this is clearly volume-preserving. In the above equality, $B$ can be seen to be an upper-triangular matrix with $1$'s on the diagonal, so it preserves volume. Hence $\overline{\Psi}$ preserves volume, as desired. 
\end{proof}

Since $F(\mathbf{r})$ is a face of $\GZ(\lambda)$, we can compute $\vol(F(\mathbf{r}))$ using Proposition~\ref{prop:face of GZ} in the language of shifted Young tableau. From Proposition~\ref{prop: volume-preserving} we also obtain a computation of $\vol(\Perm(\lambda)) = \sum_{\mathbf{r}\in \mathcal{F}_n}\vol(F(\mathbf{r}))$ in the same language. Below we illustrate some examples. 

\begin{example} 
For $(q_1, q_2, \ldots, q_{n-1}) \in \Z^{n-1}_{\geq 0}$ we introduce notation 
\[
\alpha^{(q_1,\dots,q_{n-1})}:=\frac{\alpha_1^{q_1}}{q_1!}\cdots\frac{\alpha_{n-1}^{q_{n-1}}}{q_{n-1}!}.
\]
Let $n=4$. By definition we have 
$\mathcal{F}_4=\{\mathbf{r}=(r_1,r_2,r_3)\in \Z^3\mid 1\le r_1\le 3,\ 1\le r_2\le 2, r_3=1\}$ 
which contains $6$ elements. From Proposition~\ref{prop:face of GZ} we obtain the formulas 
\begin{equation} \label{eq:volF}
\begin{split}
\vol(F(1,1,1))&=\alpha^{(3,0,0)}+\alpha^{(2,1,0)}+\alpha^{(2,0,1)}+\alpha^{(1,2,0)}+\alpha^{(1,1,1)}\\
\vol(F(2,1,1))&=\alpha^{(2,1,0)}+2\alpha^{(1,2,0)}+\alpha^{(1,1,1)}+2\alpha^{(0,3,0)}+\alpha^{(0,2,1)}\\
\vol(F(1,2,1))&=\alpha^{(2,0,1)}+\alpha^{(1,1,1)}+\alpha^{(1,0,2)}\\
\vol(F(3,1,1))&=\alpha^{(2,0,1)}+\alpha^{(1,1,1)}+\alpha^{(1,0,2)}\\
\vol(F(2,2,1))&=\alpha^{(1,2,0)}+\alpha^{(1,1,1)}+2\alpha^{(0,3,0)}+2\alpha^{(0,2,1)}+\alpha^{(0,1,2)}\\
\vol(F(3,2,1))&=\alpha^{(1,1,1)}+\alpha^{(0,2,1)}+\alpha^{(1,0,2)}+\alpha^{(0,1,2)}+\alpha^{(0,0,3)}
\end{split}
\end{equation}
which we can tabulate as follows. The columns correspond to elements of $\mathcal{F}_4$ notated for simplicity as $r_1 r_2 r_3$ instead of $(r_1, r_2, r_3)$. The entries of the tables are coefficients of the corresponding $\alpha^{(r_1, r_2, r_3)}$, where we leave an entry blank if the coefficient is $0$.

{\tiny
\begin{table}[htb]
\begin{tabular}{|c||c|c|c|c|c|c|c|c|c|c|} \hline
face $\backslash$ exponent & $300$ & $210$ & $201$ & $120$ & $111$ & $030$ & $021$ & $102$ & $012$ & $003$  \cr
\hline
\hline $F(1,1,1)$ & $1$ & $1$ & $1$ & $1$ & $1$ &       &       &       &      &            \cr
\hline $F(2,1,1)$ &       & $1$ &       & $2$ & $1$ & $2$ & $1$ &       &      &            \cr
\hline $F(1,2,1)$ &       &       & $1$ &       & $1$ &       &      & $1$ &      &            \cr
\hline $F(3,1,1)$ &       &       & $1$ &       & $1$ &       &      & $1$ &      &            \cr
\hline $F(2,2,1)$ &       &       &       & $1$ & $1$ & $2$ & $2$ &      & $1$ &            \cr
\hline $F(3,2,1)$ &       &       &       &       & $1$ &      & $1$ & $1$ & $1$ & $1$      \cr 
\hline
\end{tabular} \vspace{1.25ex}
\caption{}\label{table1}
\label{table1}
\end{table}
}

We can go further. Using Proposition~\ref{prop:x-relation} we can write $\vol(F(r_1, r_2, r_3))$ in terms of $\vol_\lambda(X^w)$ as follows: 
\begin{equation}
\begin{split}
\vol(F(1,1,1)) &=\vol_\lambda(X^{1432}),\quad \vol(F(2,1,1))=\vol_\lambda(X^{2341})+\vol_\lambda(X^{3142}),\\
\vol(F(1,2,1)) & =\vol_\lambda(X^{2413})=\vol(F(3,1,1)),\\
\vol(F(2,2,1)) &=\vol_\lambda(X^{4123})+\vol_\lambda(X^{3142}),\quad  \\
\vol(F(3,2,1)) & =\vol_\lambda(X^{3214}).
\end{split} 
\end{equation}
Combining this with~\eqref{eq:volF} we get expressions for $\vol_\lambda(X^w)$ with $\ell(w) = 3$, tabulated following the conventions for Table~\ref{table1} above. We have labelled the rows, each of which correspond to a Schubert variety $X^w$, using the one-line notation of $w$. In the last row we record the volume $\vol(\Perm(\lambda))$ of the permutohedron. 

{\tiny
\begin{table}[htb]
\begin{tabular}{|c||c|c|c|c|c|c|c|c|c|c|} \hline
 permutation $\backslash$ exponent & $300$ & $210$ & $201$ & $120$ & $111$ & $030$ & $021$ & $102$ & $012$ & $003$  \cr
\hline
\hline $1432$ & $1$ & $1$ & $1$ & $1$ & $1$  &       &      &        &       &           \cr
\hline $2341$ &       & $1$ &       & $1$ &       &       &       &       &       &            \cr
\hline $2413$ &       &       & $1$ &       & $1$ &       &       & $1$ &       &            \cr
\hline $3142$ &       &       &       & $1$ & $1$ & $2$ & $1$ &       &       &            \cr
\hline $3214$ &       &       &       &       & $1$ &       & $1$ & $1$ & $1$ & $1$      \cr
\hline $4123$ &       &       &       &       &       &      &  $1$ &       & $1$ &           \cr 
\hline$\Perm(\lambda)$ & $1$ & $2$ & $3$ & $4$ & $6$ & $4$ & $4$ & $3$ & $2$ & $1$ \cr
\hline
\end{tabular} \vspace{1.25ex}
\caption{}
\label{table2}
\end{table}
}

From Table~\ref{table2} we obtain 
\[
\begin{split}
\vol(\Perm(\lambda))=&\vol_\lambda([X^{1432}])+\vol_\lambda([X^{2341}])+2\vol_\lambda([X^{2413}])\\
&+2\vol_\lambda([X^{3142}])+\vol_\lambda([X^{3214}])+\vol_\lambda([X^{4123}])
\end{split}
\]
which corresponds to the computation in Example~\ref{exam:n=3,4Permutohedral variety}~(b). 

\end{example}

Finally, we show that the permutohedron $\Perm(\lambda)$ can be decomposed into combinatorial $(n-1)$-cubes.
We have the following.

\begin{theorem} \label{prop:perm-cubes}
The permutohedron $\Perm(\lambda)$ decomposes into $(n-1)!$ many subpolytopes $\Phi(F(\mathbf{r}))$, $\mathbf{r} \in \mathcal{F}_n$, where the vertices of the subpolytopes are all vertices of $\Perm(\lambda)$ (i.e. no extra vertices are added). Moreover, each subpolytope $\Phi(F(\mathbf{r}))$ is combinatorially an $(n-1)$-cube. 
\end{theorem}

\begin{proof} 
From Proposition~\ref{prop: volume-preserving} it follows that $\Perm(\lambda)$ decomposes into the $(n-1)!$ many images $\Phi(F(\mathbf{r}))$. Moreover, the discussion before Lemma~\ref{lemma: move b} shows that the vertices of $\Phi(F(\mathbf{r}))$ are vertices of $\Perm(\lambda)$, so no new vertices are added. 
Thus the only remaining claim that needs to be shown is the following: for any $\mathbf{r} \in \mathcal{F}_n$, the face $F(\mathbf{r})$ corresponding to $\mathbf{r}$ is combinatorially an $(n-1)$-cube. 
To see this, observe that from the definition of $F(\mathbf{r})$ (and from the description of the faces of $\GZ(\lambda)$ above) it follows that the facets of $F(\mathbf{r})$ are given by a constraint of the form 
\begin{equation} \label{eq:cube_facet}
\textup{ either } \text{$x_{r_j,j+r_j}=x_{r_j,j+r_j-1}$} \textup{  or  } \textup{$x_{r_j,j+r_j}=x_{r_j+1,j+r_j}$}
\end{equation}
for each $1 \leq j \leq n-1$. Therefore there are $2(n-1)$ many facets of $F(\mathbf{r})$. Moreover, the vertices of $F(\mathbf{r})$, being intersections of facets, must be simple vertices of $\GZ(\lambda)$ of the form described at the beginning of the section. Then each vertex of $F(\mathbf{r})$, being a simple vertex in $\GZ(\lambda)$, is also a simple vertex in the face $F(\mathbf{r})$ of $\GZ(\lambda)$. Since this holds for all vertices of $F(\mathbf{r})$, we conclude $F(\mathbf{r})$ is a simple polytope. 

Observe that any $2$-face of $F(\mathbf{r})$ is an intersection of $(n-3)$ many facets of $F(\mathbf{r})$; moreover, it is not hard to see that every $2$-face contains $4$ vertices. For a simple polytope of $\geq 3$ dimensions, it is known from \cite[Lemma 4.6]{MP} and \cite[Exercise 0.1]{Zieg} that if every $2$-face is a $4$-gon, then the polytope is combinatorially a cube. If the polytope is $2$-dimensional, the claim is obvious. This concludes the proof. 
\end{proof}

\section{The permutohedral variety and Richardson varieties}\label{section: Richardson}

As we saw in Theorem~\ref{prop:perm-cubes}, the permutohedron $\Perm(\lambda)$ decomposes into $(n-1)!$ many combinatorial $(n-1)$-cubes $F(\mathbf{r})$, $\mathbf{r} \in \mathcal{F}_n$. On the other hand, it is well-known that $\Perm(\lambda)$ is also the moment map image of the permutohedral variety $\Hess(S,h_1)$ where $h_1 = (2,3,4,\ldots,n,n)$, with respect to the standard maximal torus action on $\Fl(\C^n)$ (restricted to $\Hess(S,h_1)$). In this section we show that the combinatorial results obtained in Section~\ref{sec:decomposing}, namely, the decomposition of $\Perm(\lambda)$ into combinatorial $(n-1)$-cubes, has a geometric interpretation in terms of $\Hess(S,h_1)$.
Specifically, we consider a Richardson variety $X(\mathbf{r})$ corresponding to each $F(\mathbf{r})$ and show that the cohomology class $[\Hess(S,h_1)]$ in $H^*(\Fl(\C^n))$, is a sum of the $(n-1)!$ many cohomology classes $[X(\mathbf{r})]$ (Theorem~\ref{theorem: Hess vol as sum F}).

We begin by recalling the key fact needed for this argument, namely, the result~\eqref{eq: AT formula} of Anderson and Tymoczko. Taking $h=h_1$ in their result, we obtain 
\begin{equation} \label{eq:AT1}
[\Hess(S,h_1)]=\sum_{\substack{u,v \in \mathfrak{S}_n \\ v^{-1}u=w_{h_1} \\ \ell(u)+\ell(v)=\ell(w_{h_1})}}[X^u][X^{w_0vw_0}].
\end{equation}
To further our study of this equation we first analyze the pairs $(u,v)$ appearing in the summation on the RHS of~\eqref{eq:AT1}. First recall from Example~\ref{exam:n=3,4Permutohedral variety} that $w_h=[ n-1, n-2, \ldots, 2, 1, n]$. In particular, $w_{h_1}^{-1} = w_{h_1}$, so $v^{-1} u = w_{h_1}$ if and only if $v=uw_{h_1}$. 

\begin{lemma}\label{lem:h_1}
Let $u \in \mathfrak{S}_n$. Then
$$
\ell(u)+\ell(uw_{h_1})=\ell(w_{h_1})\Longleftrightarrow u(n)=n.
$$
In particular, there are exactly $(n-1)!$ many pairs $(u,v)$ satisfying the conditions in the summation on the RHS of~\eqref{eq:AT1}. 
\end{lemma} 

\begin{proof} 
Since $w_{h_1}=[ n-1, n-2, \ldots, 2, 1, n]$, if $u=[u(1),u(2),\dots, u(n-1),u(n)]$ then 
$uw_{h_1}=[u(n-1),\dots,u(2),u(1),u(n)]$. 

First suppose $u(n)=n$. We need to show $\ell(u)+\ell(uw_{h_1})=\ell(w_{h_1})$. For any $i \in \Z$, $1 \leq i \leq n-1$, the number of inversions in $u$ with $u(i)$ as the larger element is the number of elements in $\{u(i+1), \ldots, u(n-1)\}$ which are smaller than $u(i)$. Since $w_{h_1}$ inverts the first $n-1$ entries of $u$, the number of inversions in $uw_{h_1}$ with $u(i)$ as the larger entry is the number of elements in $\{u(i-1), \ldots, u(2),u(1)\}$ which are smaller than $u(i)$. Hence, the sum of the number of inversions in $u$ and $uw_{h_1}$ with $u(i)$ as the larger entry is exactly $u(i)-1$. By assumption $\{u(1),\ldots, u(n-1)\} = \{1,2,\ldots,n-1\}$, so 
\begin{equation} \label{eq:4-4}
\ell(u)+\ell(uw_{h_1})=\sum_{i=1}^{n-1}(u(i)-1)=(n-1)(n-2)/2=\ell(w_{h_1}).
\end{equation}

Now suppose 
$\ell(u)+\ell(uw_{h_1})=\ell(w_{h_1})$. We need to show $u(n)=n$, so suppose for a contradiction that $u(n) \neq n$, i.e., $u(n) < n$. Reasoning similar to the above yields that if $u(i) < u(n)$ then the sum of the inversions in $u$ and in $uw_h$ with $u(i)$ as the larger entry is $u(i)-1$, and if $u(i)>u(n)$, then it is $u(i)$. Since $u(n)<n$, we additionally know that there does exist an $i$, $1 \leq i \leq n-1$, with $u(i) > u(n)$. Hence by a computation similar to~\eqref{eq:4-4} we conclude 
\[
\ell(u) + \ell(u w_{h_1}) > \frac{(n-1)(n-2)}{2} = \ell(w_{h_1})
\]
which is a contradiction. Hence $u(n)=n$ as desired. 
\end{proof} 

\begin{remark} 
In fact, the proof of Lemma~\ref{lem:h_1} can be straightforwardly generalized to the case when the Hessenberg function is in ``$k$-banded form'', i.e., 
\[
h_k=(k+1,k+2,\dots,n,\dots,n)
\]
for some $k$. From this we obtain: 
$$\ell(u)+\ell(uw_{h_k})=\ell(w_{h_k})\Longleftrightarrow u(i)=i 
$$
for $n-k+1\le i\le n$. 
\end{remark}

Let $\mathbf{r} = (r_1, \ldots, r_{n-1}) \in \mathcal{F}_n$. Following the methods outlined in Section~\ref{sec:decomposing}, we can identify the vertices of $F(\mathbf{r})$ with a subset of $\mathfrak{S}_n$, and the Bruhat-maximal element $\mathbf{r}_{max}$ in the image of $F(\mathbf{r})$ is the one which specifies (in addition to the equations already given in~\eqref{eq: equations for Fr}) the equations 
\[
x_{r_j, j+r_j} = x_{r_j, j+r_j-1} \quad \textup{ for } 1 \leq j \leq n-1.
\]
Correspondingly, the Bruhat-minimal element $\mathbf{r}_{min}$ is the one specifying 
\[
x_{r_j, j+r_j} = x_{r_j+1, j+r_j} \quad \textup{ for } 1 \leq j \leq n-1.
\]

\begin{lemma}\label{lemma: rmin and rmax formula} 
We have $\mathbf{r}_{min}(n)=n$ and $\mathbf{r}_{max} = \mathbf{r}_{min}(s_{n-1} s_{n-2} \cdots s_1)$. In terms of the one-line notation 
\begin{equation}\label{eq: rmin and rmax}
\mathbf{r}_{min} = [ a_1, \cdots, a_{n-1}, n] \quad \textup{ and } \quad \mathbf{r}_{max} = [n, a_1, \cdots, a_{n-1}]
\end{equation} 
where $\{a_1, \ldots, a_{n-1}\} = [n-1]$ as sets. Moreover, the association $\mathcal{F}_n \to \mathfrak{S}_{n-1}$ given by $\mathbf{r} \mapsto \mathbf{r}_{min}$ is a bijective correspondence. 
\end{lemma}

\begin{proof} 
We induct on $n$. For the base case $n=2$, the claim is obvious.  Now suppose by induction that the claim is true for $n-1$. Define an element $\mathbf{r}' \in \mathcal{F}_{n-1}$ by setting $\mathbf{r}' = (r_2, r_3, \ldots, r_{n-1})$, or in other words, $r'_i := r_{i+1}$ for $1 \leq i \leq n-2$. (Visually, this corresponds to ``deleting the vertices in the bottom diagonal'' in the dot diagram corresponding to $\mathbf{r}$, deleting any edges adjacent to those vertices, and then interpreting what remains as the diagram for $\mathbf{r}'$.) 
By induction we know that 
\[
\mathbf{r}'_{min} = [a_1, \ldots, a_{n-2}, n-1], \quad \mathbf{r}'_{max} = [n-1, a_1, \ldots, a_{n-2}]
\]
where $\{a_1, \ldots, a_{n-2}\} = \{1,2,\ldots,n-2\} = [n-2]$. From the definition of $\mathbf{r}_{min}$ and $\mathbf{r}_{max}$ given above and the correspondence between face diagrams and permutations as explained at the beginning of Section~\ref{sec:decomposing} it now follows that, using the notation $r_1 =k$, we have 
\begin{equation*}
\begin{split} 
\mathbf{r}_{min} &= [a_1+1, \ldots, a_{k-1}+1, 1, a_{k}+1, \ldots, a_{n-2}+1, n], \\
\mathbf{r}_{max} &= [n, a_1+1, \ldots, a_{k-1}+1, 1, a_{k}+1, \ldots, a_{n-2}+1],
\end{split} 
\end{equation*}
where the $1$ in the one-line notation of $\mathbf{r}_{max}$ is in the $k+1$-st spot. In particular,~\eqref{eq: rmin and rmax} holds.

Finally, since we have seen that $\mathbf{r}_{min}(n)=n$ we may view it as an element in $\mathfrak{S}_{n-1}$. Moreover, from the construction of the elements $\mathbf{r}_{min}$ it follows that any element in $\mathfrak{S}_{n-1}$ can be realized as $\mathbf{r}_{min}$ for some $\mathbf{r}$. From this we conclude that the association $\mathbf{r} \mapsto \mathbf{r}_{min}$ is a bijection between $\mathcal{F}_n$ and $\mathfrak{S}_{n-1}$.  This completes the proof. 
\end{proof}

Also note that since $\mathbf{r}_{min} < \mathbf{r}_{max}$ in Bruhat order, it makes sense to define the corresponding \textbf{Richardson variety} 
\begin{equation}\label{eq: richardson for rmin rmax} 
X(\mathbf{r}) := X^{\mathbf{r}_{min}} \cap X_{\mathbf{r}_{max}}.
\end{equation}
These varieties play a role in our description of the cohomology class corresponding to the permutohedral variety, as below.

\begin{theorem}\label{theorem: Hess vol as sum F}
Let $\Hess(S, h_1)$ be the permutohedral variety in $\mathrm{Flag}(\C^n)$. Then, considered as a cohomology class in $H^*(\mathrm{Flag}(\C^n))$, we have 
\begin{equation}\label{eq: Hess vol as sum F}
[\Hess(S,h_1)]=\sum_{\mathbf{r}\in\mathcal{F}_n}[X(\mathbf{r})].
\end{equation}
\end{theorem}

\begin{proof} 
Starting from the result of Anderson and Tymoczko given in~\eqref{eq:AT1}, Lemma~\ref{lem:h_1} together with the remark before Lemma~\ref{lem:h_1} imply that 
\[
[ \Hess(S, h_1)] = \sum_{u \in \mathfrak{S}_{n-1}} [X^u] [X^{w_0 uw_{h_1} w_0}].
\]
Recall also that, in general, we have $[X^{w_0 v w_0}] = [X_{v w_0}]$ for any $v \in \mathfrak{S}_n$. Applying this to the case $v=u w_{h_1}$ in the equation above we obtain 
\[
[\Hess(S,h_1)] = \sum_{u \in \mathfrak{S}_{n-1}} [X^u] [X_{ u w_{h_1} w_0}].
\]
Now recall that $w_{h_1}$ is the full inversion in $\mathfrak{S}_{n-1}$, i.e. $w_{h_1} = [n-1, n-2, \ldots, 2, 1, n]$, so 
\[
w_{h_1} w_0 = [ n, 1, 2, \ldots, n-2, n-1] = s_{n-1} s_{n-2} \cdots s_1.
\]
From Lemma~\ref{lemma: rmin and rmax formula} it now follows that 
\begin{equation*}
\begin{split} 
[\Hess(S, h_1)] &=  \sum_{\mathbf{r} \in \mathcal{F}_n}  [X^{\mathbf{r}_{min}}] [X_{\mathbf{r}_{max}}]\\
 & = \sum_{\mathbf{r} \in \mathcal{F}_n} [ X^{\mathbf{r}_{min}} \cap X_{\mathbf{r}_{max}} ] \\
  & = \sum_{\mathbf{r} \in \mathcal{F}_n} [ X(\mathbf{r}) ] \\
  \end{split} 
  \end{equation*} 
 where the second equality is a well-known fact about Schubert varieties (see e.g. \cite{brio}) and the last equality is by the definition of $X(\mathbf{r})$. This completes the proof. 
\end{proof}

\begin{remark} 
It is known that the Richardson variety $X(\mathbf{r})$ defined above is in fact a compact smooth toric variety which is a Bott tower (Bott manifold) \cite{le-ma-pa19}. It is also not hard to see that 
$\vol_\lambda(X(\mathbf{r}))=\vol(F(\mathbf{r}))$.  Indeed, when expanding $[X^{\mathbf{r}_{min}}][X_{\mathbf{r}_{max}}]$ in terms of the formula of Theorem~\ref{theorem: KST}, it is straightforward to see that the faces $F\cap F^*$ appearing on the RHS of the equation include the $F(\mathbf{r})$; thus, it follows that $\vol_\lambda(X(\mathbf{r})) \geq \vol(F(\mathbf{r}))$. On the other hand, we also know
\[
\vol_\lambda(\Hess(S,h_1)) = \sum_{\mathbf{r} \in \mathcal{F}_n} \vol_\lambda(X(\mathbf{r}))
\]
from Theorem~\ref{theorem: Hess vol as sum F}. The fact that $\Hess(S,h_1)$ is a toric variety with moment map image precisely $\Perm(\lambda)$ implies that $\vol_\lambda(\Hess(S,h_1)) = \vol(\Perm(\lambda))$. Thus from Proposition~\ref{prop: volume-preserving} we obtain 
\[
\vol_\lambda(\Hess(S,h_1)) = \sum_{\mathbf{r} \in \mathcal{F}_n} \vol(F(\mathbf{r}))
\]
and hence 
\[
\sum_{\mathbf{r} \in \mathcal{F}_n} \vol_\lambda(X(\mathbf{r})) = 
 \sum_{\mathbf{r} \in \mathcal{F}_n} \vol(F(\mathbf{r}))
 \]
 from which it follows that $\vol_\lambda(X(\mathbf{r}))=\vol(F(\mathbf{r}))$. From this we can see that 
 Theorem~\ref{theorem: Hess vol as sum F} is a precise geometric analogue of Proposition~\ref{prop: volume-preserving}, since by taking volume of both sides of~\eqref{eq: Hess vol as sum F}, we obtain the equality $\vol(\Perm(\lambda)) = \sum_{\mathbf{r} \in \mathcal{F}_n} \vol(F(\mathbf{r}))$. 
\end{remark}

 We conclude with a problem which naturally arises out of the considerations above. 
It is well-known \cite[Exercise 12, p.180]{fult97} that there exist non-negative integers $a_w \geq 0$ such that 
\begin{equation}\label{eq-prob}
\displaystyle{[\Hess(S,h_1)]=\sum_{\substack{w\in \mathfrak{S}_n\\ \ell(w)=n-1}}a_w[X^{w_0w}]}
\end{equation}
where $w_0$ is the longest element in $\mathfrak{S}_n$. We can consider the following. 

\begin{problem} \label{conjecture:positivity}
Are the coefficients in~\eqref{eq-prob} \emph{strictly positive}?
\end{problem}

\begin{remark}\label{remark: kaji} 
It can be checked that the answer to Problem~\ref{conjecture:positivity} is ``Yes'' when $n=3$ and $n=4$.
Let $Y_3$, respectively $Y_4$, denote the permutohedral varieties in $\Fl(\C^3)$, respectively $\Fl(C^4)$. Then 
\begin{align*}
[Y_3]&=[X^{213}]+[X^{132}] \\
[Y_4]&=[X^{1432}]+[X^{2341}]+2[X^{2413}]+2[X^{3142}]+[X^{3214}]+[X^{4123}]
\end{align*}
where we use the simplified one-line notation $w=w(1)w(2)\ldots w(n)$ for permutations in $\mathfrak{S}_n$.
In fact, we have computer-based evidence \cite{Kaj} that Problem~\ref{conjecture:positivity} is true when $n \leq 6$. 
We can also consider Problem~\ref{conjecture:positivity} for other Lie types.
Again, we have computational evidence \cite{Kaj} the coefficients are strictly positive for the Lie types $B_3$, $B_4$, $B_5$, $C_3$, $D_4$, $F_4$, $G_2$. 
\end{remark}

\end{document}